\documentclass{amsart}

\usepackage{amssymb}
\usepackage{latexsym}
\usepackage{amsmath}
\usepackage{euscript}
\usepackage{graphics}
\usepackage{graphicx}
\usepackage{tikz}
\usetikzlibrary{spy}
\graphicspath{ {./images/} }
\usepackage{enumerate}

      \def\sL{{\mathfrak L}}

      \def\dC{{\mathbb C}}

      \def\dR{{\mathbb R}}

\def\cD{{\mathcal D}}      
      
      \def\cL{{\mathcal L}}

      \def\cU{{\mathcal U}}

\def\bm\chi{\mbox{\boldmath$\chi$}}

\def\diag{{\rm diag\,}}

\let\xker=\ker \def\ker{{\xker\,}}

\def\Im{\operatorname{Im}}

\newtheorem{theorem}{Theorem}[section]
\newtheorem{proposition}[theorem]{Proposition}
\newtheorem{corollary}[theorem]{Corollary}

\theoremstyle{remark}
\newtheorem{example}[theorem]{Example}
\newtheorem{remark}[theorem]{Remark}

\numberwithin{equation}{section}

\usepackage[colorinlistoftodos]{todonotes}

\newcommand{\tc}{\textcolor{black}}

\begin{document}

\title{Complex Jacobi matrices generated by Darboux transformations}


\author[R.~Bailey]{Rachel~Bailey}
\address{
RB,
Department of Mathematics\\
University of Connecticut\\
341 Mansfield Road, U-1009\\
Storrs, CT 06269-1009, USA}
\email{rachel.bailey@uconn.edu}

\author[M.~Derevyagin]{Maxim~Derevyagin}
\address{
MD,
Department of Mathematics\\
University of Connecticut\\
341 Mansfield Road, U-1009\\
Storrs, CT 06269-1009, USA}
\email{maksym.derevyagin@uconn.edu}


\dedicatory{\tc{Dedicated to the 70th anniversary of Paco Marcell\'an}}

\begin{abstract} 
In this paper, we study complex Jacobi matrices obtained by the Christoffel and Geronimus transformations at a nonreal complex number, \tc{including the properties of the corresponding sequences of orthogonal
polynomials}. \tc{We also present some invariant and semi-invariant properties of Jacobi matrices under such transformations. For instance,} we show that a Nevai class is invariant under \tc{the transformations in question, which is not true in general, and that the ratio asymptotic still holds outside the spectrum of the corresponding symmetric complex Jacobi matrix but the spectrum could include one extra point}. In principal, these transformations can be iterated and, for example, we demonstrate how Geronimus transformations can lead to $R_{II}$-recurrence relations, which in turn are related to \tc{orthogonal rational functions} and pencils of Jacobi matrices.
\end{abstract}
\maketitle

\section{Introduction}

We denote by $\dC$ the set of all complex numbers and by $\dC_{\pm}=\{z\in\dC|\pm\Im z>0\}$ the upper and lower half-planes, respectively. Let $\mathcal{L}$ be a complex-valued linear functional defined on the vector space $\mathbb{C}[z]$ of all polynomials with complex coefficients. Evidently, such 
a functional is uniquely determined by its moments
\[
s_j=\mathcal{L}[z^j]\in\dC, \quad j=0, 1, 2, \dots.
\]
If the moments are such that 
\[
\det(s_{i+j})_{i,j=0}^n\ne 0, \quad n=0, 1, 2, \dots
\]
then $\mathcal{L}$ is called {\it quasi-definite} or {\it regular}. It is known that $\mathcal{L}$ is quasi-definite if and only if there exists a sequence of polynomials $P_n(z)$ of degree $n$ that are orthogonal with respect to $\cL$. The latter means that they satisfy the following relations
\[
\cL[P_n(z)P_m(z)]=0, \quad n\ne m,
\]
and
\begin{equation*}
\cL[P_n^2(z)]\ne 0, \quad n=0, 1, 2,\dots.    
\end{equation*}
This {\it Orthogonal Polynomial System} (in what follows it will be abbreviated to OPS) is unique provided that the sequence of leading coefficients of $P_n(z)$'s is fixed. To be definite, we assume that $P_n(z)$ is a monic polynomial for each $n$ and in this case the OPS verifies a three-term recurrence relation of the form
\begin{align}
\label{def:rec}P_n(z) = (z-c_n)P_{n-1}(z)-\lambda_nP_{n-2}(z), \quad n=2, 3, \dots,
\end{align} or, equivalently,
\begin{align}
\label{def:rec2}
zP_n(z)=P_{n+1}(z)+c_{n+1}P_n(z)+\lambda_{n+1}P_{n-1}(z), \quad n=1, 2, \dots,
\end{align}
where $P_0(z)=1$ and $P_{1}(z)=z-c_1$. Moreover, the complex numbers $c_n$ and $\lambda_n$ can be computed in terms of $\mathcal{L}$ as follows:
\begin{equation}\label{lAndcFormulas}
\lambda_{n+1} = \frac{\mathcal{L}[P^2_n(z)]}{\mathcal{L}[P^2_{n-1}(z)]},\quad c_{n+1}= \frac{\mathcal{L}[zP^2_{n}(z)]}{\mathcal{L}[P^2_{n}(z)]},
\quad n=0,1,2,\dots,    
\end{equation}
where we assume that $\mathcal{L}\{P_{-1}^2(z)\}=1$ for consistency. Note that $\lambda_n\ne0$ for $n=1,2,3,\dots$ since $\mathcal{L}$ is quasi-definite (for more details about the basic theory of quasi-definite linear functionals see \cite[Chapter I]{Chihara}).

\tc{It is without any doubt that one of the most famous OPSs is the Chebyshev polynomials. Recall that the monic Chebyshev polynomials $\{{V}_n(z)\}_{n=0}^{\infty}$ of the third kind form an OPS with respect to the linear functional (e.g. see \cite{MH03})
\[
\mathcal{L}[p(x)] = \int_{-1}^1 p(x) \sqrt{\frac{1+x}{1-x}}\,{dx}
\]     
and they satisfy the three-term recurrence relation
\[
z{V}_{n}(z) ={V}_{n+1}(z)+\frac{1}{4}{V}_{n-1}(z), \quad n=1,2,\dots
\]
with the initial conditions
\[
{V}_0(z) =1, \quad {V}_1(z)=z-\dfrac{1}{2}.
\]
Clearly, the linear functional for the Chebyshev polynomials is defined by the measure $ \sqrt{\dfrac{1+x}{1-x}}\,{dx}$ on $(-1,1)$ and in this case we say that the polynomials are orthogonal with respect to the measure. Moreover, if $c_n\in\dR$ and $\lambda_n>0$, then, according to the Favard theorem, the underlying functional is defined by a positive measure $d\mu$, 
in which case for $\kappa\in\dR$ the transformation
\[
d\mu(x)\to d\mu^*(x)=(x-\kappa)d\mu(x)
\]
defines a new OPS $\{P_n^*(\kappa,z)\}_{n=0}^\infty$ as long as $x-\kappa$ does not change the sign on the convex hull of the support of the measure $d\mu$. This transformation is called the Christoffel transformation at $\kappa$. It has an inverse transformation, which is given by the formula
\[
d\mu(x)\to d\mu^{-*}(x)=\frac{1}{x-\kappa}d\mu(x)+M\delta_{\kappa},
\]
where $M$ is a real number and $\delta_{\kappa}$ is the delta function supported at $\kappa$, and which is called the Geronimus transformation at $\kappa$.
These two transformations give rise to a family of discrete dynamical systems defined by iterations of the forms
\begin{equation}\label{def:ddt}
\begin{split}
(\{c_n\}_{n=2}^\infty,\{\lambda_n\}_{n=2}^\infty)&\to (\{c_n^*(\kappa)\}_{n=2}^\infty,\{\lambda_n^*(\kappa)\}_{n=2}^\infty), \\
(\{c_n\}_{n=2}^\infty,\{\lambda_n\}_{n=2}^\infty)&\to (\{c_n^{-*}(\kappa)\}_{n=2}^\infty,\{\lambda_n^{-*}(\kappa)\}_{n=2}^\infty),
\end{split}
\end{equation}
where $c_n^*$, $\lambda_n^*$ and $c_n^{-*}$, $\lambda_n^{-*}$ correspond to $d\mu^*$ and $d\mu^{-*}$, respectively (note that the existence of the resulting sequences $c_n^*$, $\lambda_n^*$ and $c_n^{-*}$, $\lambda_n^{-*}$ is not automatically guaranteed). For instance, given a sequence of points $\kappa_1$, $\kappa_2$, $\kappa_3$, \dots, one can define the $k$-th evolution to be
\begin{equation}\label{def:ddt_evolution}
 c_n\to (c_n^*(\kappa_{2k-1}))^{-*}(\kappa_{2k}), \quad   \lambda_n\to (\lambda_n^*(\kappa_{2k-1}))^{-*}(\kappa_{2k}),
\end{equation}
provided each transformation is correctly defined.
The Chebyshev polynomials play a very special role for such discrete dynamical systems. Namely, the composition of a Christoffel transformation and a Geronimus transformation
\[
d\mu(x)=\sqrt{\dfrac{1+x}{1-x}}\,{dx}\to d\widetilde{\mu}(x)=\dfrac{1-x}{1+x}\sqrt{\dfrac{1+x}{1-x}}\,{dx}=\sqrt{\dfrac{1-x}{1+x}}\,{dx}
\]
maps the monic Chebyshev polynomials $\{{V}_n(z)\}_{n=0}^{\infty}$ of the third kind into the monic Chebyshev polynomials $\{{W}_n(z)\}_{n=0}^{\infty}$ of the fourth kind. The latter satisfies the three-term recurrence relation
\[
z{W}_{n}(z) ={W}_{n+1}(z)+\frac{1}{4}{W}_{n-1}(z), \quad n=1,2,\dots
\]
with the initial conditions
\[
{W}_0(z) =1, \quad {W}_1(z)=z+\dfrac{1}{2}.
\]
This demonstrates that the pair of sequences
\begin{equation}\label{RefrenceSeq}
c_2=0, c_3=0,\dots, \quad \lambda_2=\dfrac{1}{4}, \lambda_3=\dfrac{1}{4}, \dots    
\end{equation}
is a fixed point of the transformation composed of the two given in \eqref{def:ddt} for different points and $M=0$, and so it is an equilibrium solution to the corresponding discrete dynamical system in which the state of the system evolves according to \eqref{def:ddt_evolution}. 
Up to an alteration in the initial data, one can see that the sequence \eqref{RefrenceSeq} is also a fixed point for the two-iterated Christoffel transformation
\[
\dfrac{dx}{\sqrt{1-x^2}}\to (1-x)(1+x)\dfrac{dx}{\sqrt{1-x^2}}={\sqrt{1-x^2}}\,dx,
\]
which maps the Chebyshev polynomials of the first kind onto the Chebyshev polynomials of the second kind.
In this light, the statement of \cite[Theorem 3.1]{Simon04} basically reads that specific iterations of the double Christoffel transformations of an appropriate measure $d\mu$ converge to $\dfrac{dx}{\sqrt{1-x^2}}$ weakly, which generates the Chebyshev polynomials of the first kind and is also related to the sequence \eqref{RefrenceSeq}. Thus \cite[Theorem 3.1]{Simon04} can be interpreted as a stability result for this equilibrium solution. This observation gives a warrant for a further investigation of the general discrete dynamical systems in question and the stability of their equilibrium and periodic solutions. In this paper, we begin this study by exploring the analytic nature of Chrtistoffel and Geronimus transformations, the building blocks of the dynamical systems, in the case when we lose positivity and so the situation does not fall under the classical settings. When there is no need to distinguish between the two, these transformations are referred to as Darboux transformations \cite{SZh95}, \cite{BM04} and sometimes they are also referred to as commutation methods \cite{D78}, \cite{GT96}. 
To be more specific about our goal, let us remind that one can associate \eqref{def:rec2} to
the following monic Jacobi matrix 
\[
J_m=\begin{pmatrix}
c_1 & 1 & 0 & \cdots   \\
     \lambda_2 & c_2 & 1 &   \\
     0 & \lambda_3 & c_3 & 1 \ddots  \\
     \vdots & & \ddots & \ddots
\end{pmatrix},
\]
where the subscript $m$ stands for monic. So, what we do in this paper is we study how Darboux transformations affect the analytic properties of a real Jacobi matrix $J_m$ corresponding to a positive measure and its symmetrization, which includes the analytic properties of the corresponding orthogonal polynomials, when $\kappa$ is a nonreal number.
In particular, we establish some invariant and semi-invariant properties of Jacobi matrices under such transformations. Unlike the algebraic properties of Darboux transformations, which have been extensively studied (see \cite{BM04}, \cite{Y02}, \cite{Z97}, and the references therein), the effects of Christoffel and Geronimus transformations at $\kappa\in\dC\setminus\dR$ on analytic properties are not addressed in the existing literature. Besides, the Darboux transformations in questions do not preserve the realness of the Jacobi matrix and so we are basically studying certain families of complex Jacobi matrices, which, in the sense of dynamical systems, are elements of orbits of real Jacobi matrices. Note that in recent years there has been a growth of interest in complex Jacobi matrices (for example, see \cite{BC04}, \cite{BGK09}, \cite{SF17}, and \cite{Sw20}) and their applications in computational mathematics \cite{B01} (also see the references therein) and in non-classical quantum mechanics \cite{GGKN}, \cite{Z}.}

The paper is organized as follows. Section 2 gives a brief refresher of the \tc{case when the linear functional is positive-definite} and presents some auxiliary statements. Next, in Sections 3 and 4 we thoroughly \tc{analyze} the  Christoffel and Geronimus transformations at $\kappa\in\dC_\pm$ 
\tc{using theory of orthogonal polynomials}. \tc{Since the resulting measure is no longer positive-definite, it is not obvious if we can iterate such transformations. Therefore, we will then establish conditions under which we can perform two successive iterations of both transformations.} After that, Section 5 discusses the spectral properties of the Darboux transformations of real Jacobi matrices at $\kappa\in\dC_\pm$. Finally, \tc{in Section 6 we show how Darboux transformations give rise to orthogonal rational functions and the underlying three-term recurrence relations that correspond to $R_I$- and $R_{II}$-continued fractions, which were introduced in \cite{IsmailMasson95}.}

\section{Preliminaries: the positive-definite case}

Recall that the functional $\mathcal{L}$ is called positive-definite if $\cL[p(x)]>0$ for every polynomial $p(x)$ that is not identically zero and is non-negative for all real $x$. It is not so hard to see that given a non-negative function $w(x)$ on the interval $(a,b)$, the functional 
\[
\cL[p(x)]=\int_a^bp(x)w(x)\,dx
\]
gives an example of a positive-definite linear functional provided that $w(x)$ is integrable on $(a,b)$ and $w(x)>0$ on a subset of $(a,b)$ of positive Lebesgue measure. Also, in the same way, any probability measure that is  compactly supported on $\dR$ defines a positive-definite functional (to find out more details about the positive-definite case one can consult either \cite{Chihara} or \cite{Ismail09}). 

If $\mathcal{L}$ is positive-definite, all the moments $s_j=\mathcal{L}(z^j)$ are real and therefore, the coefficients $c_n$ and $\lambda_n$ of the the three-term recurrence relation 
\[
zP_n(z)=P_{n+1}(z)+c_{n+1}P_n(z)+\lambda_{n+1}P_{n-1}(z), \quad n=1, 2, \dots
\]
are also real according to  \eqref{lAndcFormulas} and the fact that a monic OPS with respect to a positive definite linear functional must be real. Furthermore, the positive-definiteness of $\cL$ implies that 
\begin{equation*}
\cL[P_n^2(z)]>0, \quad n=0, 1, 2,\dots    
\end{equation*}
and thus by \eqref{lAndcFormulas}, we get that $\lambda_n>0$ for $n=1,2,3,\dots$.
Another consequence of positive-definiteness is that the zeros of $P_n(z)$ are simple and real. Also, it is well known that in this case the zeros of $P_{n+1}(z)$ and $P_n(z)$ interlace. These facts yield properties that we will need and 
\tc{we prove them in the following statement for the reader's convenience.}
\begin{proposition}
\label{prop:zeros}
Let $\{P_n(z)\}_{n=1}^\infty$ be a monic OPS with respect to a positive definite linear functional. 
\begin{enumerate}[(i)]
\item If $z \in \dC_+$ then 
\[
0>\Im\left(\frac{P_{n-1}(z)}{P_n(z)}\right)\geq -\frac{1}{\Im z}.
\]
\item If $z \in \dC_-$ then 
\[
-\frac{1}{\Im z}\geq \Im\left(\frac{P_{n-1}(z)}{P_n(z)}\right)>0.
\]
\item If $z \in \dC_\pm$ is fixed and, in addition, the sequences $\lambda_n$ and $c_n$ are bounded then the sequences 
\[
\frac{P_{n-1}(z)}{P_n(z)}, \quad \frac{P_{n+1}(z)}{P_n(z)}
\]
are bounded as well.
\end{enumerate}
\end{proposition}
\begin{proof}
 Let $x_{n,j}$ denote the $j$-th zero of $P_n(z)$. Since the zeros of $P_n(z)$ are simple and real, $\dfrac{P_{n-1}(z)}{P_n(z)}$ has a partial fraction decomposition of the form
\begin{equation}
\label{def:partialfrac}\frac{P_{n-1}(z)}{P_n(z)}=\frac{\alpha_{n,1}}{z-x_{n,1}}+\frac{\alpha_{n,2}}{z-x_{n,2}}+\dots +\frac{\alpha_{n,n}}{z-x_{n,n}}.
\end{equation}
Notice that in this partial fraction decomposition, 
\begin{align*} P_{n-1}(z) &= \alpha_{n,1}(z-x_{n,2})\dots(z-x_{n,n})+\alpha_{n,2}(z-x_{n,1})(z-x_{n,3})\dots(z-x_{n,n})+ \dots\\
&+ \alpha_{n,n}(z-x_{n,1})(z-x_{n,2})\dots(z-x_{n, n-1}).
\end{align*}
Thus, since $P_{n-1}(z)$ is monic, we have that $\sum_{i=1}^n \alpha_{n,i} =1$ for all $n =1, 2, \dots$.\\
Now, let $z \in \dC_+$. Then  $\Im\left(\frac{1}{z-x_{n,j}} \right)<0$ for all $j=1, 2, \dots, n$. From (\ref{def:partialfrac}) we get that
\[
\alpha_{n,j} = \lim\limits_{z \rightarrow x_{n,j}}(z-x_{n,j})\frac{P_{n-1}(z)}{P_n(z)} = \frac{P_{n-1}(x_{n,j})}{P'_{n}(x_{n,j})},
\]
where $\alpha_{n,j}>0$ for all $j=1,2, \dots, n$ since the zeros of $P_{n+1}(z)$ and $P_n(z)$ interlace \cite[cf. Chapter I, Theorem 5.3]{Chihara}.
Therefore, $\Im\left(\frac{\alpha_{n,j}}{z-x_{n,j}}\right)<0$ for all $z \in \dC_+$, $j=1, \dots, n$ and thus, from (\ref{def:partialfrac}),
\[
\Im\left(\frac{P_{n-1}(z)}{P_n(z)}\right)<0.
\]
Next, notice that if $z \in \mathbb{C}_+$, we have $\Im\left( \frac{1}{z-x_{n,j}}\right)\geq -\frac{1}{\Im z }$. Hence, we arrive at 
\[
0 > \Im \left( \frac{P_{n-1}(z)}{P_n(z)}\right)\geq -\frac{1}{\Im z}\left(\sum_{i=1}^n \alpha_{n,i}\right) = -\frac{1}{\Im z}.
\] 
Since $\overline{\Big(\frac{P_{n-1}(z)}{P_n(z)}\Big)}= \frac{P_{n-1}(\overline{z})}{P_n(\overline{z})}$, $(ii)$ is a direct consequence of $(i)$.
To prove $(iii)$, one needs to observe that
\[
\frac{P_{n+1}(z)}{P_n(z)}=z-c_{n+1}-\lambda_{n+1}\frac{P_{n-1}(z)}{P_n(z)},
\]
and that $\left| \frac{P_{n-1}(z)}{P_n(z)}\right| \leq \frac{1}{\left|\Im{z}\right|}$ due to \eqref{def:partialfrac}, which yields the desired result.
\end{proof}

We will also need another family associated to $\mathcal{L}$. Namely, let us consider polynomials $Q_n(z)$ that are defined via the formula 
\begin{equation}\label{def:Qn}
Q_n(z)=\mathcal{L}\left(\frac{P_n(z)-P_n(y)}{z-y}\right)_y, \quad n \geq 0,
\end{equation}
where the subscript $y$ indicates that the functional acts on the variable $y$. The $Q_n(z)$ are indeed polynomials of degree $n-1$ and they are called polynomials of the second kind or numerator polynomials (see \cite{Ismail09}).
Rewriting equation (\ref{def:rec2}) as 
\begin{equation}\label{DiffEqGeneral}
    zy_n=y_{n+1}+c_{n+1}y_n+\lambda_{n+1}y_{n-1}
\end{equation} 
we have a second-order liner difference equation in the variable $n$ that has two  linearly independent solutions. Clearly, one of these solutions is the OPS  $\{P_n(z)\}_{n=0}^\infty$ and it is easy to check that $y_n=Q_n(x)$ satisfies the same second-order difference equation subject to the initial conditions 
\[
Q_0(z)=0,\quad Q_1(z)=1. 
\]
Thus, $\{Q_n(z)\}_{n=0}^\infty$ is linearly independent of $\{P_n(z)\}_{n=0}^\infty$ and so it is the second solution. Also, in the positive-definite case the zeros of $P_n(z)$ and $Q_n(z)$ interlace. 
Let us stress here that if $\mathcal{L}$ is positive-definite then by Favard's theorem, 
$\{Q_n(z)\}_{n=1}^\infty$ 
\tc{is an} OPS with respect to some positive-definite linear functional. Finally, if the entries of equation (\ref{def:rec2}) are such that $c_n \rightarrow c$ and $\lambda_n\rightarrow a$ for $c\in \mathbb{R}$ and $a \in [0, \infty)$, then $\{P_n(z)\}_{n=0}^\infty$ (or the corresponding Jacobi matrix) is said to be in the {\it Nevai class } $\mathcal{N}(a,c)$. Since $\{Q_n(z)\}_{n=0}^\infty$ satisfy the same recurrence relation as $\{P_n(z)\}_{n=0}^\infty$, it is clear that if $\{P_n(z)\}_{n=0}^\infty$ is in the Nevai class $\mathcal{N}(a,c)$, then so is $\{Q_n(z)\}_{n=0}^\infty$. 

\section{Christoffel transformation}

\tc{In this section we consider Christoffel transformation and we discuss some properties of the transformed polynomials and Jacobi matrices. In particular, we demonstrate that under certain conditions the boundedness of Jacobi matrices as well as the ratio asypmtotics are preserved under Christoffel transformation. Note that such properties do not hold in general as can be seen from the findings presented in \cite{D15}.} 

Let $\mathcal{L}$ be a \tc{positive-definite} linear functional and let $\kappa \in \mathbb{C}\setminus\dR$. Define the {\it Christoffel transformation} of $\mathcal{L}$ at $\kappa$ as a linear functional $\mathcal{L}^{*}$ such that for a polynomial $p(z)$,
\[
\mathcal{L}^*[p(z)] =\mathcal{L}[(z-\kappa)p(z)].
\]
\tc{In this case, $P_n(\kappa) \neq 0$ for any integer $n \geq 0$ and we can} define the corresponding kernel polynomials by the formula 
\[P_n^*(\kappa, z) = (z-\kappa)^{-1}\left[P_{n+1}(z) - \frac{P_{n+1}(\kappa)}{P_n(\kappa)}P_n(z) \right], \quad n=0,1,2, \dots.
\]
According to \cite[Theorem I.7.1]{Chihara}, $\{P_n^*(\kappa,z)\}_{n=0}^\infty$ is a monic OPS with respect to $\mathcal{L}^*$. It is worth stressing here that if $\kappa\in\dC_\pm$, $\mathcal{L}^*$ is not positive-definite but it is possible to get some information about the corresponding \tc{Jacobi matrix} and OPS $\{P_n^*(\kappa,z)\}_{n=0}^\infty$. For example, using the recurrence relation in (\ref{def:rec2}) and the fact that any finite number of elements of the sequence $\{P_n(z)\}_{n=0}^\infty$ form a linearly independent set, we have that the sequence $\{P^*_n(\kappa,z)\}_{n=0}^\infty$ satisfies the following three-term recurrence relation:
\[
zP^*_{n}(\kappa,z)= P^*_{n+1}(\kappa,z)+c^*_{n+1}(\kappa)P^*_{n}(\kappa,z)+\lambda^*_{n+1}(\kappa)P^*_{n-1}(\kappa,z),
\] 
where
\begin{equation}\label{lAndcFormulasCT}
\lambda^*_{n+1}(\kappa) = \lambda_{n+1}\frac{P_{n+1}(\kappa)P_{n-1}(\kappa)}{P^2_n(\kappa)} , \quad c^*_{n+1}(\kappa)= c_{n+2}-\frac{P_{n+1}(\kappa)}{P_n(\kappa)}+\frac{P_{n+2}(\kappa)}{P_{n+1}(\kappa)}.    
\end{equation}
\tc{Thus, the underlying monic Jacobi matrix is 
\[
J^*_m(\kappa)=\begin{pmatrix}
c^*_1(\kappa) & 1 & 0 & \cdots   \\
     \lambda^*_2(\kappa) & c^*_2(\kappa) & 1 &   \\
     0 & \lambda^*_3(\kappa) & c^*_3(\kappa) & 1 \ddots  \\
     \vdots & & \ddots & \ddots
\end{pmatrix},
\]}
\tc{where the subscript $m$ stands for monic. From formulas \eqref{lAndcFormulasCT} we get that boundedness is preserved under the Christoffel transformation at $\kappa\in\dC\setminus\dR$. In what follows we will omit the $\kappa$-dependence when it is clear from the context and we will call $J_m^*$ the Christoffel transformation of $J_m$.}

\tc{Before we proceed with the properties of $J_m^*$ and the corresponding polynomials, let us consider an example.}

\begin{example}\label{FibJac}
\tc{Recall that the monic Chebyshev polynomials $\{{U}_n(x)\}_{n=0}^{\infty}$ of the second kind form an OPS with respect to the linear functional 
\[
\mathcal{L}[p(x)] = \int_{-1}^1 p(x) \sqrt{1-x^2}\,dx
\]     
and they satisfy the three-term recurrence relation
\[{U}_{n+1}(x) = x{U}_{n}(x)-\frac{1}{4}{U}_{n-1}(x), \quad n=1,2,3,\dots
\]
with the intial conditions
\[
{U}_0(x) =1, \quad {U}_1(x)=x.
\]
These polynomials are related to the Fibonacci sequence via the formula
\[
F_n=\frac{2^nU_n(i/2)}{i^n}, \quad n=0, 1, 2, \dots,
\]
where $F_0=1$, $F_1=1$, $F_2=2$, $F_3=3$, \dots is the Fibonacci sequence. Thus, setting $\kappa=i/2$, \eqref{lAndcFormulasCT} yields
\[
\lambda_{n+1}^*=\frac{1}{4}\frac{F_{n+1}F_{n-1}}{F_n^2}, \quad
c_{n+1}^*=\frac{i}{2}\frac{F_{n+2}F_n-F_{n+1}^2}{F_{n}F_{n+1}},
\]
which taking into account the relation $F_{n+1}F_{n-1}-F_n^2=(-1)^{n-1}$ reduce to 
\[
\lambda_{n+1}^*=\frac{(-1)^{n-1}}{4F_n^2}+\frac{1}{4}, \quad
c_{n+1}^*=i\frac{(-1)^n}{2F_{n}F_{n+1}}.
\]
The underlying monic matrix Jacobi is clearly a complex Jacobi matrix and is the simplest representative of complex Jacobi matrices we consider in this paper as many families of orthogonal polynomials can be explicitly evaluated at a given complex number.}
\end{example}

\begin{proposition}\label{prop:boundednessC} \tc{Let $J_m$ be the monic Jacobi matrix corresponding to a monic OPS $\{P_n(z)\}_{n=0}^\infty$ that is generated by a positive-definite linear functional. Assume that $J_m$ is bounded, that is, its entries $\lambda_n$ and $c_n$ are bounded and let $\kappa\in\dC\setminus\dR$. Then $J^*_m$ is bounded and as a result the set of all zeros of the polynomials $P^*_n(\kappa,z)$'s is bounded.}
\end{proposition}
\begin{proof}
From Proposition \ref{prop:zeros} and formulas \eqref{lAndcFormulasCT} one concludes that the sequences $\lambda_n^*$ and $c_n^*$ are also bounded. Thus, \tc{the Jacobi operator $J^*_m$ is bounded in $\ell^2$. Then, the boundedness of zeros follows from the fact that for a bounded complex Jacobi matrix the set of all zeros of the corresponding polynomials are contained in a bounded convex set (e.g. see \cite[Theorem 3.4 (a)]{B01}).} 
\end{proof}

In the case of a Nevai class, one can say a bit more.

\begin{proposition}\label{prop:Nevai_Pn*}
If a monic Jacobi matrix $J_m$ is in the Nevai class $\mathcal{N}(a,c)$, then the Christoffel transformation $J^*_m$ of $J_m$ at $\kappa\in\dC\setminus\dR$ is also in the same Nevai class $\mathcal{N}(a,c)$. 
\end{proposition}
\begin{proof}
\tc{Note that if } $\{P_n(z)\}_{n=0}^\infty$ is in $\mathcal{N}(a,c)$, \tc{ then for any $z\in\dC\setminus\dR$ }we have \tc{ the following {\it ratio asymptotic}:}
\begin{equation}\label{def:RatioAsympF}
 \frac{P_{n+1}(z)}{P_n(z)}\rightarrow f(z):= \frac{(z-c)+\sqrt{(z-c)^2-4a^2}}{2},    
\end{equation}
where we take the branch of the square root such that $\sqrt{\dots}=z+O(\frac{1}{z})$ near $z=\infty$ (e.g. see \cite{Simon04} or \cite{Nevai79}). Thus, it follows from \eqref{lAndcFormulasCT} that $\lambda^*_{n} \rightarrow a$ and $c^*_n\rightarrow c$.
\end{proof} 



\begin{remark}
\tc{It should be emphasized that the condition $\kappa\in\dC\setminus\dR$ is cruicial here. Indeed, the Stahl's counterexample shows that the Christoffel transformation of the Chebyshev polynomials at some real points leads to an unbounded Jacobi matrix (for details see  \cite{D15}). In other words, when $\kappa$ is real, neither boundedness nor the Nevai class have to be invariant under the Christoffel transformation at such $\kappa$. Still, it is true for some real $\kappa$'s, in which case it leads to a real Jacobi matrix.}
\end{remark}
In the positive-definite case the zeros of orthogonal polynomials are real. However, the Christoffel transformation at $\kappa$ does not preserve positive-definiteness but we can still get some estimates \tc{on the location of the zeros} for the corresponding OPS.
\begin{theorem}
\label{thrm: zeros}
Let $\{P_n(z)\}_{n=0}^\infty$ be a monic OPS with respect to a positive-definite linear functional. Let $\kappa\in\dC_\pm$. Then for the  corresponding kernel polynomials $\{P^*_n(\kappa,z)\}_{n=0}^\infty$ and $n\geq 1$ we have that
\begin{enumerate}[(i)]
    \item if $\kappa \in \mathbb{C}_+, \text{ then the zeros of } P_n^*(\kappa, z) \text{ lie in the horizontal strip }
    \newline \left\{ z\in \mathbb{C}\, \middle\vert\,  0 < \Im z \leq -\frac{1}{\Im\left(\frac{P_{n-1}(\kappa)}{P_n(
\kappa)}\right)}\right\}$;
\item if $\kappa \in \mathbb{C}_-, \text{ then the zeros of } P_n^*(\kappa, z) \text{ lie in the horizontal strip } 
\newline \left\{z\in \mathbb{C} \,\middle\vert \, -\frac{1}{\Im\left(\frac{P_{n-1}(\kappa)}{P_n(
\kappa)}\right)}\leq \Im z <0\right\}.$
\end{enumerate}

\end{theorem}
\begin{proof}
Let $\kappa$ and $z_0$ be such that $\kappa \in \mathbb{C}_+$ and $z_0\in \mathbb{C}_- $. Assume by contradiction that $P_n^*(\kappa,z_0) =0$ for some $n \geq 1$. Then
\begin{equation}\label{eq:kernel1}
P_{n+1}(z_0)=\frac{P_{n+1}(\kappa)}{P_n(\kappa)}P_n(z_0). 
\end{equation}
Since $P_n(z)$ has only real zeros, (\ref{eq:kernel1}) can be rewritten as
\begin{equation}\label{eq:kernel2} \frac{P_{n}(z_0)}{P_{n+1}(z_0)} = \frac{P_{n}(\kappa)}{P_{n+1}(\kappa)}.
\end{equation}
By Proposition \ref{prop:zeros} we know $\Im\left(\frac{P_{n}(\kappa)}{P_{n+1}(\kappa)}\right)<0$ thus from (\ref{eq:kernel2}), $\Im\left( \frac{P_{n}(z_0)}{P_{n+1}(z_0)}\right)<0$. Next, applying Proposition \ref{prop:zeros} again, we must have $\Im\left(\frac{P_{n}(z_0)}{P_{n+1}(z_0)}\right)>0$ which is a contradiction. Since $z_0$ was arbitrary, we see $P_n^*(\kappa,z)$ has no zeros in $\mathbb{C}_-$ for any $n =1,2,\dots$.\\
\indent Now suppose $x_0 \in \mathbb{R}$ and there exists some $n\geq 1$ such that $P_n^*(\kappa, x_0) =0$. If $x_0$ is not a zero of $P_n(z)$, then (\ref{eq:kernel2}) holds, yet $\Im\left(\frac{P_{n}(x_0)}{P_{n+1}(x_0)}\right) =0$ while $\Im\left(\frac{P_{n}(\kappa)}{P_{n+1}(\kappa)} \right) <0$ which is a contradiction. If $x_0$ is a zero of $P_n(z)$ then by the separation theorem for the zeros, $P_{n+1}(x_0) \neq 0$, hence  (\ref{eq:kernel1}) cannot hold. Thus $P_n^*(\kappa, z)$ has no real zeros for any $n=1,2 \dots$.\\
\indent Finally, let $z_0 \in \mathbb{C}_+$ such that $\Im z_0> -\frac{1}{\Im\left(\frac{P_{n-1}(\kappa)}{P_n(\kappa)}\right)}$ and suppose $z_0$ is a zero of $P^*_n(\kappa,z)$ for some $n \geq 1$. Then $-\frac{1}{\Im z_0} > \Im\left(\frac{P_{n-1}(\kappa)}{P_n(\kappa)}\right)$ and $\Im\left(\frac{P_{n-1}(\kappa)}{P_n(\kappa)}\right)= \Im\left(\frac{P_{n-1}(z_0)}{P_n(z_0)}\right)$. However, by Proposition $\ref{prop:zeros}$, we have \newline
$\Im\left(\frac{P_{n-1}(z_0)}{P_n(z_0)}\right)<-\frac{1}{\Im z_0} \leq\Im\left(\frac{P_{n-1}(z_0)}{P_n(z_0)}\right)$. Thus, the zeros of $P^*_n(\kappa, z)$ can only lie in $\left\{z \in \mathbb{C}: 0 < \Im z \leq -\frac{1}{\Im\left(\frac{P_{n-1}(\kappa)}{P_n(
\kappa)}\right)}\right\}$. This proves $(i)$.\\
\indent Since $\{P_n(z)\}_{n=0}^\infty$ is a sequence of real polynomials, $\overline{\frac{P_n(z_0)}{P_{n+1}(z_0)}}= \frac{P_n(\overline{z_0})}{P_{n+1}(\overline{z_0})}$ so $(ii)$ follows from $(i)$.
 
\end{proof} 

\begin{remark}
\tc{It is well known that the zeros of orthogonal polynomials are the eigenvalues of the corresponding finite truncations of the Jacobi matrix. In this light, Theorem \ref{thrm: zeros} gives an estimate for such eigenvalues. Later, in Section 5, we will discuss what happens with the spectrum of semi-infinite real Jacobi matrices under Christoffel transformation.}
\end{remark}

\begin{example}\label{Ex:Chebyshev} Recall that the monic Chebyshev polynomials $\{{T}_n(x)\}_{n=0}^{\infty}$ of the first kind form an OPS with respect to the linear functional 
\[
\mathcal{L}[p(x)] = \int_{-1}^1 p(x) (1-x^2)^{-\frac{1}{2}}dx
\]     
and they satisfy the three-term recurrence relation
\[{T}_{n+1}(x) = x{T}_{n}(x)-\frac{1}{4}{T}_{n-1}(x), \tc{\quad n=2,3,\dots}
\]
with \tc{the first ones given by }
\[
{T}_0(x) =1, \quad {T}_1(x)=x, \quad \tc{T_2(x)=x^2-1/2.}
\]
According to Theorem \ref{thrm: zeros}, for the monic Chebyshev polynomials $\{{T}_n(x)\}_{n=0}^{\infty}$ of the first kind, the zeros of the corresponding kernel polynomials
\[
{T}_n^*(\kappa,z)=(z-\kappa)^{-1}\left[{T}_{n+1}(z)-\frac{{T}_{n+1}(\kappa)}{{T}_{n}(\kappa)}{T}_n(z) \right]
\]
lie in $\dC_+$ provided that $\kappa\in\dC_+$. Furthermore, using Mathematica we can see that when $\kappa=i$ and $\kappa=1+i$ the zeros get closer to the real line when $n$ is increasing (see Figures \ref{fig1} and \ref{fig2}).
\begin{figure}[h!]
\centering
\begin{minipage}{.5\textwidth}
  \centering
   \includegraphics[scale=0.36]{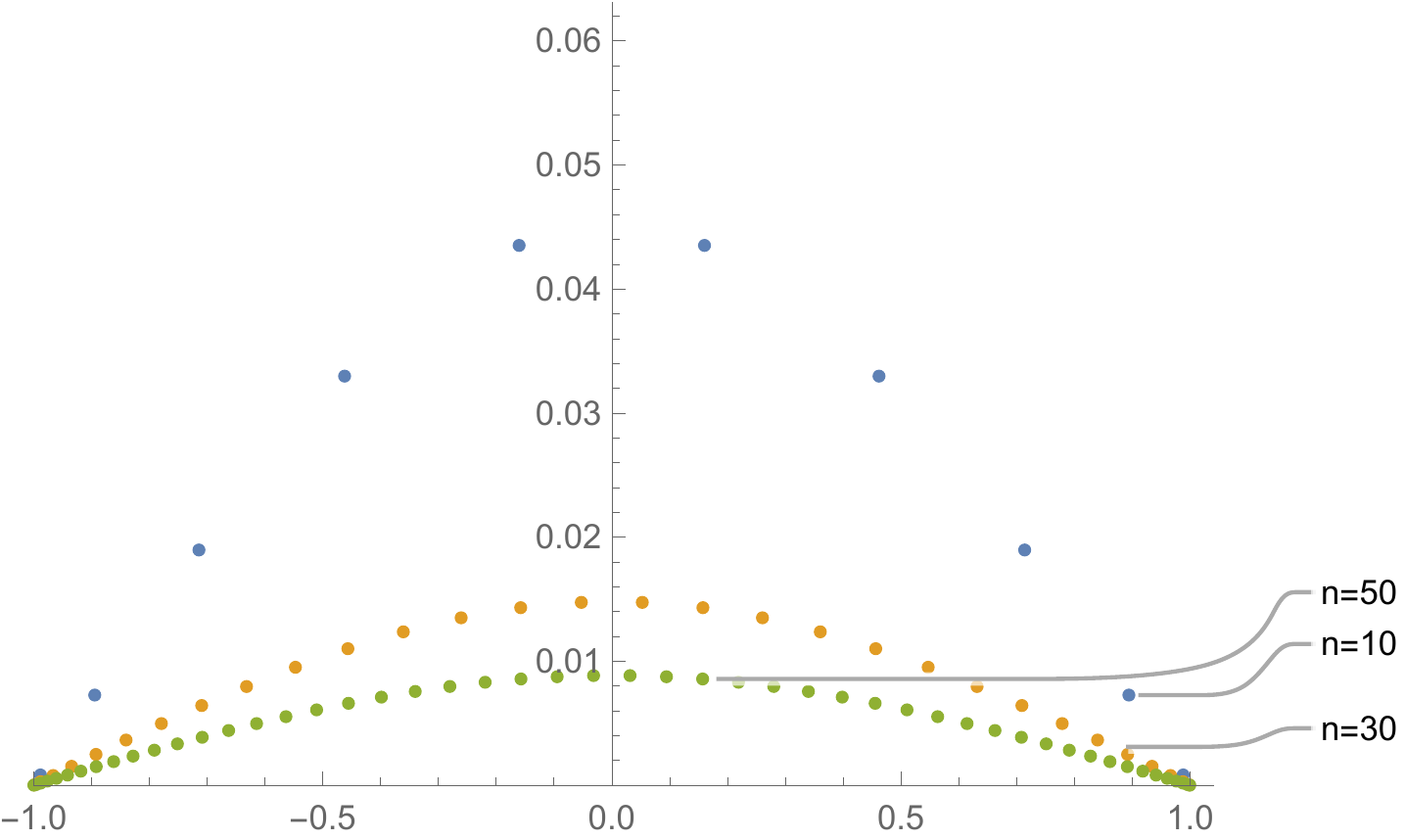}
    \caption{The behavior of the zeros of ${T}_n^*(i,z)$ when $n$ is increasing.}
    \label{fig1}
\end{minipage}%
\begin{minipage}{.5\textwidth}
  \centering
  \includegraphics[scale=0.36]{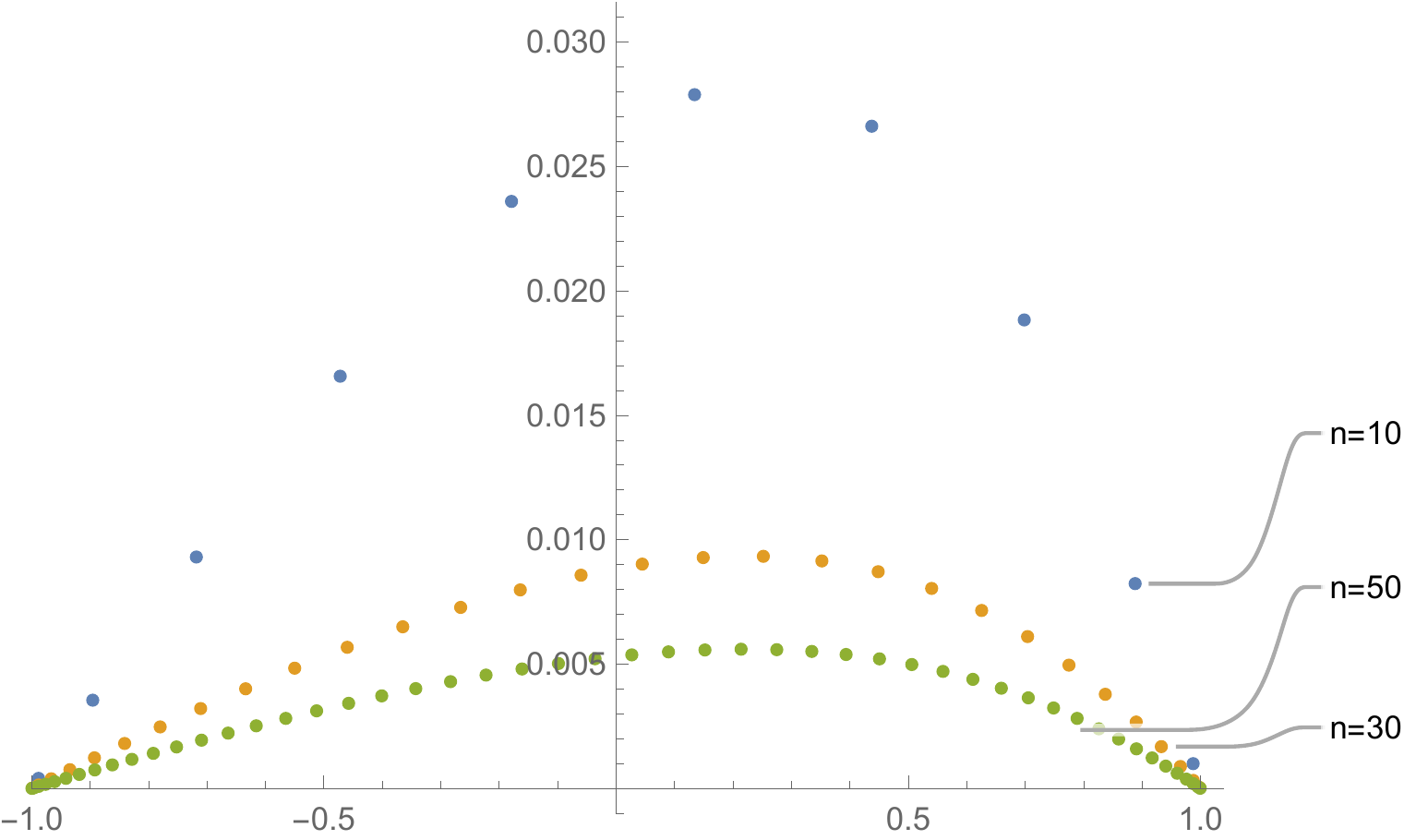}
    \caption{The behavior of the zeros of ${T}_n^*(1+i,z)$ when $n$ is increasing.}
    \label{fig2}
\end{minipage}
\end{figure}
\end{example}
It turns out that this is a typical behavior for a large class of kernel polynomials. 
\begin{theorem}
\label{thrm: n-zero behavior} Let $\{P_n(z)\}_{n=0}^\infty$ be a monic OPS with respect to a positive-definite linear functional and with kernel polynomials $\{P^*_n(\kappa,z)\}_{n=0}^\infty$. If $\{P_n(z)\}_{n=0}^\infty$ is in the Nevai class $\mathcal{N}(a,c)$, then the imaginary part of the zeros of the polynomial $P^*_n(\kappa,z)$ converges to zero as $n$ approaches infinity.
\end{theorem}
\begin{proof}
Assume $\kappa \in \mathbb{C}_+$ and let $z_{n,1}, z_{n,2}, \dots, z_{n,n}$ be the zeros of $P^*_n(\kappa, z)$. Note by Theorem $\ref{thrm: zeros}$, we have $\Im z_{n,j}>0$ for all $n=1,2,\dots$. Let $z^*_n$ be such that $\Im z^*_n =\max\{\Im z_{n,j}: z_{n,j} \text{ is the $j$-th zero of $P^*_n(\kappa,z)$}\}$ and let $G_n(z)= \frac{P_{n+1}(z)}{P_n(z)}-\frac{P_{n+1}(\kappa)}{P_n(\kappa)}$.  Since $\{P_n(z)\}_{n=0}^\infty$ is in the Nevai class $\mathcal{N}(a,c)$,  $G_n(z)$ converges uniformly to $f(z)-f(\kappa)$ where $f(z) = \frac{(z-c)+\sqrt{(z-c)^2-4a^2}}{2}$ and the square root is taken with $\sqrt{\dots}=z+O(\frac{1}{z})$ near $z=\infty$ (see \cite[Theorem 2.1]{Simon04}). Since $f(z)$ is injective in $\mathbb{C}_+$, $f(z)-f(\kappa)$ has a simple, isolated zero at $\kappa$. Let $\Delta$ be a sufficiently small neighborhood of $\kappa$ so that $f(z)-f(\kappa)\neq 0$ in $\Delta\setminus \{\kappa\}$.
By Hurwitz's Theorem, $G_n(z)$ has the same number of zeros in $\Delta$ as $f(z)-f(\kappa)$ for sufficiently large $n$ so $G_n(z)$ has only a simple zero in $\Delta$ for large $n$. Since $G_n(\kappa)=0$ for all 
\tc{$n=1,2,\dots$}, it must be the case that $\kappa$ is a simple zero of $G_n(z)$ for large $n$. Therefore, since  $P^*_n(\kappa, z)= \frac{1}{z-\kappa}P_n(z)G_n(z)$ and $(z-\kappa)$ divides $P_n(z)G_n(z)$, we have $P^*_n(\kappa, \kappa)\neq 0$ for large $n$. Thus, for large $n$, the zero set of $G_n(z)$ is $\{z_{n,j}\}_{j=1}^n\sqcup \{\kappa\}$ \footnote[1]{$\sqcup$ denotes the disjoint union}.\\
\indent Since the zeros $\{z^*_n\}_{n=1}^\infty$ lie in a compact set due to Proposition \ref{prop:boundednessC}, there exists a convergent subsequence $\{z^*_{n_m}\}_{m=1}^\infty$. Suppose $z^*_{n_m}\rightarrow z_0$ for some $z_0\in \mathbb{C}_+$. Then $G_{n_m}(z^*_{n_m}) \rightarrow f(z_0)-f(\kappa)$ so by the injectivity of $f$, we must have $z_0 = \kappa$. Now since $P^*_n(\kappa, z)$ has the the same zeros as $G_n(z)$ except $\kappa$, $P^*_n(\kappa,z)$ has no zeros in $\Delta$ for large $n$, contradicting the fact that $z^*_{n_m}\rightarrow \kappa$. Thus, $z_0$ must be real. Since $\Im z^*_{n_m}\rightarrow 0$ and  $\Im z^*_{n_m} =\max\{\Im z_{n_m,j}: z_{n_m,j} \text{ is the $j$-th zero of $P^*_{n_m}(\kappa,z)$}\}$ we must have that the imaginary part of the zeros of $P^*_n(\kappa,z)$ converge to 0.\\
\indent The case when $\kappa \in \mathbb{C}_-$ follows similarly.
\end{proof}

Also, we know that in the case of Nevai class the ratio of two consecutive orthogonal polynomials converges and it turns out that this ratio asymptotic is preserved under the Christoffel transformation at $\kappa\in\mathbb{C}\setminus \mathbb{R}$.
\begin{theorem} \label{prop:pn_star_ratio}
Let $\{P_n(z)\}_{n=0}^\infty$ be a monic OPS with respect to a \tc{positive-definite linear functional and let $\{P_n^*(\kappa,z)\}_{n=0}^\infty$ be the corresponding kernel polynomials for some $\kappa\in\mathbb{C}\setminus \mathbb{R}$. 
If $\{P_n(z)\}_{n=0}^\infty$ is in a Nevai class then 
\[
\lim_{n\to\infty}\frac{P_{n+1}^*(\kappa, z)}{P_{n}^*(\kappa,z)}=\lim_{n\to\infty}\frac{P_{n+1}(z)}{P_n(z)}=f(z),
\]
where the ratio converges on compact subsets of $\mathbb{C}\setminus \mathbb{R}$ and $f$ is defined in \eqref{def:RatioAsympF}. }
\end{theorem}
\proof Without loss of generality, let $\kappa \in \mathbb{C}_+$ and let $G_n(z) = \frac{P_{n+1}(z)}{P_n(z)}-\frac{P_{n+1}(\kappa)}{P_n(\kappa)}$ as in the proof of Theorem \ref{thrm: zeros}. Then rewriting $\frac{P^*_{n}(\kappa,z)}{P^*_{n-1}(\kappa,z)}$, we have
\begin{equation}\label{eq:pn_star_ratio}\frac{P^*_{n}(\kappa,z)}{P^*_{n-1}(\kappa,z)} = \frac{P_n(z)}{P_{n-1}(z)}\frac{G_n(z)}{G_{n-1}(z)}.
\end{equation}
Since $G_n(z)$ has a simple zero at $z=\kappa$, we can write $G_n(z)= (z-\kappa)g_n(z)$ for a rational function $g_n(z)$.  Since $\{P_n(z)\}_{n=0}^\infty$ is in a Nevai class, $G_n(z)\rightarrow f(z)-f(\kappa)$ uniformly on compact subsets of $\mathbb{C}\setminus \mathbb{R}$ \tc{as was mentioned in the proof of Proposition \ref{prop:Nevai_Pn*}}. Now let $\epsilon>0$ and fix $\delta>0$ such that the circle $|z-\kappa|=\delta$ lies entirely in $\mathbb{C}_+$ \tc{and} let $K$ be a compact subset of $\mathbb{C}_+$. By the uniform convergence of $G_n(z)$, we know there exists $N>0$ such that  \[\left|g_n(z)-\frac{f(z)-f(\kappa)}{z-\kappa}\right|<\epsilon\] for all $n \geq N$ and for all $z \in K \cap \{z \in \mathbb{C}: |z-\kappa|\geq \delta\}$, hence $g_n(z)$ converges uniformly to $\frac{f(z)-f(\kappa)}{z-\kappa}$ on compact subsets of $\mathbb{C}_+$ that do not contain $\kappa$. In particular, $\left|g_n(z)-\frac{f(z)-f(\kappa)}{z-\kappa}\right|<\epsilon$ on $|z-\kappa|=\delta$. Since $\frac{f(z)-f(\kappa)}{z-\kappa}$  has a removable singularity at $\kappa$, $\left| g_n(z)-\frac{f(z)-f(\kappa)}{z-\kappa}\right|<\epsilon$ inside the disk $\{z\in \mathbb{C}:|z-\kappa|<\delta\}$ for all $n \geq N$ by the Maximum Principle. Therefore, $g_n(z)$ converges uniformly to $\frac{f(z)-f(\kappa)}{z-\kappa}$ on compact subsets of $\mathbb{C}_+$. Notice that $\overline{g_n(z)}=g_n(\overline{z})$ and $\overline{f(z)}= f(\overline{z})$, so  $g_n(z)$ converges uniformly to $\frac{f(z)-f(\kappa)}{z-\kappa}$ on compact subsets of $\mathbb{C}\setminus \mathbb{R}$. The same holds for $G_{n-1}(z)=(z-\kappa)g_{n-1}(z).$ \\
\indent Recall by Theorem \ref{thrm: n-zero behavior}, the zeros of $g_n(z)$ and $g_{n-1}(z)$ shrink to the real line, so re-writing equation (\ref{eq:pn_star_ratio}) as
\[
    \frac{P^*_{n}(\kappa,z)}{P^*_{n-1}(\kappa,z)} = \frac{P_n(z)}{P_{n-1}(z)}\frac{g_n(z)}{g_{n-1}(z)}
\]
we have that  $\frac{P^*_{n}(\kappa,z)}{P^*_{n-1}(\kappa,z)}$ converges uniformly to $f(z)$ on compact subsets of $\mathbb{C}\setminus \mathbb{R}$. 
\qed
\begin{remark}\label{remark:Christoffel_iterations}
\tc{The new essential part of Theorem \ref{prop:pn_star_ratio} is the observation that the ratio asymptotic is preserved under the Christoffel transformation at $\kappa\in\mathbb{C}\setminus \mathbb{R}$, which opens up a new perspective even for a real $\kappa$. To be specific, it is clear that the Chebyshev polynomials are in the Nevai class $\mathcal{N}(1/4,0)$. Moreover, one can easily establish that the ratio asymptotic holds true for the Chebyshev polynomials outside $[-1,1]$.  In fact, one can see that Theorem \ref{prop:pn_star_ratio} holds true for any $\kappa$ outside the support of the orthogonality measure for $\{P_n(z)\}_{n=0}^\infty$. Thus, the consecutive applications of Proposition \ref{prop:Nevai_Pn*} and Theorem \ref{prop:pn_star_ratio} show that a measure of the form
\[
\prod_{k=1}^N(x-\kappa_k)\frac{dx}{\sqrt{1-x^2}}, \quad \kappa_k\in\dR\setminus[-1,1], 
\]
which is an element of an orbit of $\dfrac{dx}{\sqrt{1-x^2}}$ with respect to one of the discrete dynamical systems in question, is also in the Nevai class $\mathcal{N}(1/4,0)$. Therefore, from this perspective the Denisov-Rakhmanov theorem in the case of purely absolutely continuous measure is just the limit case. This idea will be given some rigor and will be further developed in the case of complex Jacobi matrices elsewhere.} 
\end{remark}

\tc{\begin{remark} As is known, for a complex Jacobi matrix from a Nevai class, the ratio asymptotic holds true for the corresponding polynomials (see \cite{B01} and references therein). However, one must exclude a discrete set of points some of which could be in $\mathbb{C}\setminus\mathbb{R}$ and this set needs to be determined, which is not always an easy task. The fact that we have a ratio asymptotic on $\mathbb{C}\setminus\mathbb{R}$ uses the specifics of $J_m^*$. On top of that, it suggests that the Jacobi matrix corresponding to the kernel polynomials should not have non-real spectrum, which will be rigorously proved in Section 5.  
\end{remark}}

\tc{If one wants to implement the idea described in Remark \ref{remark:Christoffel_iterations} for complex Jacobi matrices, one has to understand general iterations of Christoffel transformations, in which case it is not clear for what choice of such transformations an OPS exists. Thus the first question would be to give a reasonable description of the cases when we can guarantee the existence of an OPS. Let us concentrate on the case of the second iteration and consider the iterated functional
\[
\mathcal{L}^{**}[f(z)] = \mathcal{L}[(z-\kappa_2)(z-\kappa_1)f(z)].
\]
In this case we can prove the following extension of the Christoffel theorem to the non-positive definite case} (see \cite[Chapter 2.2.7]{Ismail09} for more details about the classical Christoffel theorem).
\begin{theorem}\label{the:2ItEx}
\tc{Let $\cL$ be a positive-definite linear functional and let $d\mu$ be the corresponding measure, that is, 
\[
\cL[p]=\int_{\dR}p(x)d\mu(x).
\]
Also, let $\kappa_1\in\dC_{\pm}$ and $\kappa_2\in\dC_{\mp}$. Then the OPS $\{P^{**}_n(\kappa_1,\kappa_2, z)\}_{n=0}^{\infty}$ with respect to $\mathcal{L}^{**}$ exists. Equivalently, the complex measure
\[
(x-\kappa_1)(x-\kappa_2)d\mu(x)
\]
generates a system of monic orthogonal polynomials $P^{**}_n(\kappa_1,\kappa_2, z)$ and, thus, a monic Jacobi matrix $J^{**}_m(\kappa_1,\kappa_2)$. In particular, we have that 
\[
\Delta_n(\kappa_1,\kappa_2)=\begin{vmatrix}
P_{n+1}(\kappa_1)&P_n(\kappa_1)\\
P_{n+1}(\kappa_2)&P_n(\kappa_2)
\end{vmatrix}\ne 0, \quad n=0, 1, 2, \dots
\]
and 
\[
P^{**}_n(\kappa_1,\kappa_2, z)=\frac{1}{(z-\kappa_1)(z-\kappa_2)}
\frac{\begin{vmatrix}
P_{n+2}(\kappa_1)&P_{n+1}(\kappa_1)&P_n(\kappa_1)\\
P_{n+2}(\kappa_2)&P_{n+1}(\kappa_2)&P_n(\kappa_2)\\
P_{n+2}(z)&P_{n+1}(z)&P_n(z)
\end{vmatrix}}{\Delta_n(\kappa_1,\kappa_2)}.
\]}
\end{theorem}
\begin{proof} \tc{To begin with, note that $P_n(z)$ and $P_{n+1}(z)$ have only real zeros. 
Therefore, we have
\[
\Delta_n(\kappa_1,\kappa_2)=P_{n+1}(\kappa_1)P_{n+1}(\kappa_2)
\left[\frac{P_n(\kappa_2)}{P_{n+1}(\kappa_2)}-\frac{P_n(\kappa_1)}{P_{n+1}(\kappa_1)}\right],
\]
which is not zero due to Proposition \ref{prop:zeros}. 
 We thus see that the polynomials
\begin{equation}\label{def:Pn**}
P^{**}_n(\kappa_1,\kappa_2, z) = (z-\kappa_2)^{-1}\left[P^*_{n+1}(\kappa_1, z) - \frac{P^*_{n+1}(\kappa_1,\kappa_2)}{P^*_n(\kappa_1,\kappa_2)}P^*_n(\kappa_1,z) \right]
\end{equation}
are correctly defined for any nonnegative integer $n$. As a result, the orthogonality is immediate. 
Then 
\begin{multline}
    P^{**}_n(\kappa_1,\kappa_2, z) = (z-\kappa_2)^{-1}\left[P^*_{n+1}(\kappa_1, z) - \frac{P^*_{n+1}(\kappa_1,\kappa_2)}{P^*_n(\kappa_1,\kappa_2)}P^*_n(\kappa_1,z) \right]\\
    =\frac{1}{z-\kappa_2}\Big[\frac{1}{z-\kappa_1}\left(P_{n+2}(z)-\frac{P_{n+2}(\kappa_1)}{P_{n+1}(\kappa_1)}P_{n+1}(z) \right)-\\
    \frac{P^*_{n+1}(\kappa_1, \kappa_2)}{P^*_n(\kappa_1, \kappa_2)} \frac{1}{z-\kappa_1}\left(P_{n+1}(z)-\frac{P_{n+1}(\kappa_1)}{P_n(\kappa_1)}P_n(z) \right)\Big]\\
    =\frac{1}{(z-\kappa_1)(z-\kappa_2)}\Big[P_{n+2}(z)-\left(\frac{P_{n+2}(\kappa_1)P_n(\kappa_2)-P_{n+2}(\kappa_2)P_n(\kappa_1)}{P_{n+1}(\kappa_1)P_n(\kappa_2)-P_{n+1}(\kappa_2)P_n(\kappa_1)}\right)P_{n+1}(z)+\\
    +\left(\frac{P_{n+2}(\kappa_1)P_{n+1}(\kappa_2)-P_{n+2}(\kappa_2)P_{n+1}(\kappa_1)}{P_{n+1}(\kappa_1)P_n(\kappa_2)-P_{n+1}(\kappa_2)P_n(\kappa_1)}
    \right)P_n(z)\Big]
    \end{multline}
and the expression in the square brackets can be recast as the determinant divided by $\Delta_n(\kappa_1,\kappa_2)$.}
\end{proof} 
\begin{remark}
\tc{If $\kappa_1=\overline{\kappa}_2\in\dC\setminus\dR$, we have that $(x-\kappa_1)(x-\kappa_2)>0$ for all $x\in\dR$ and thus the statement of the theorem reduces to the classical Christoffel theorem. However, when $\kappa_1\ne\overline{\kappa}_2$, the quadratic polynomial $(x-\kappa_1)(x-\kappa_2)$ is not positive on the real line and the fact that the resulting functional for the specified choice of $\kappa$'s is quasi-definite is new.}
\end{remark}

\tc{Since the boundedness is preserved under the Christoffel transformation provided we choose the points appropriately, starting with a bounded monic Jacobi matrix $J_m$ one can pick $\kappa_3$ to be outside the numerical range of $J^{**}_m(\kappa_1,\kappa_2)$ and so on. The latter is not easy to find explicitly in the general situation and so it would be nice to find a generalization of Theorem \ref{the:2ItEx} for 3 and more points, which would provide us with a universal way of picking the points for consecutive iterations.}

\section{Geronimus transformation}

In this section we will consider a transformation that is inverse to the Christoffel transformation at $\kappa\in\dC_{\pm}$ \tc{and we will mostly follow the same scheme we implemented in the previous section}. Let $\mathcal{L}$ be a complex-valued linear functional. Define its {\it Geronimus transformation} at $\kappa$ as a linear functional $\mathcal{L}^{-*}$ whose Christoffel transformation at $\kappa$ is $\mathcal{L}$, that is, $(\mathcal{L}^{-*})^*=\mathcal{L}$. More precisely, the definition reads \tc{as follows}
\[
\mathcal{L}^{-*}((z-\kappa)p(z))=\mathcal{L}(p(z)),
\]
where $p(z)$ is any polynomial. It is not so hard to see from the above relation that for any polynomial $p(z)$ we have that
\[
\mathcal{L}^{-*}(p(z))=
\mathcal{L}\left(\dfrac{p(z)-p(\kappa)}{z-\kappa}\right)+p(\kappa)\mathcal{L}^{-*}(1),
\]
where $\mathcal{L}^{-*}(1)$ is not uniquely determined by the definition and therefore it can be an arbitrary constant.

From the point of view of orthogonality, it is sometimes more convenient to have forms (for instance, see \cite{DM14}) and given a linear functional $\mathcal{L}$ one can actually define a bilinear form. More precisely, for two polynomials $p(z)$ and $q(z)$ we define 
\[ 
(p,q)_0 =\mathcal{L}(p(z)q(z)). 
\]
In \tc{the same way}, $\mathcal{L}^{-*}$ generates the bilinear form $[\cdot, \cdot]_1$ that satisfies 
\[
[(t-\kappa)p, q]_1=[p, (t-\kappa)q]_1=(p,q)_0 
\] \tc{\text{ for the real variable }t.}

In case the given \tc{ linear} functional has an explicit representation of the form
\begin{equation}\label{functional:IntRepr}
 \mathcal{L}(p(t))=\int_a^b p(t)\,d\mu(t),   
\end{equation}
where $d\mu(t)$ is a positive measure, whose support is contained in the interval $[a,b] \tc{\subset \mathbb{R}}$, one can also be more specific about its Geronimus transformation. 
\begin{proposition}
Let $\mathcal{L}$ be of the form \eqref{functional:IntRepr} and let $\kappa\in\dC_{\pm}$. Then the Geronimus transformation $\mathcal{L}^{-*}$ of $\mathcal{L}$ at $\kappa$ corresponds to the bilinear form $[\cdot,\cdot]_1$ that admits the representation
\begin{equation}\label{eq:bilinear}
 [p,q]_1=\int_a^b p(t)q(t)\frac{d\mu(t)}{t-\kappa}+\left(s_0^*-\int_a^b\frac{d\mu(t)}{t-\kappa}\right)p(\kappa)q(\kappa),\quad p,q\in \mathbb{C}[z],
\end{equation}
where $s_0^*$ is an arbitrary complex number.
\end{proposition}
\begin{proof}
\tc{The proof follows like that of Proposition 2.2 in \cite{DM14} with the substitution $t \rightarrow t-\kappa$, but note that $t-\kappa$ is no longer real}.
\end{proof}

\tc{It is important to note that although} $s_0^*$ can be an arbitrary complex number, not all numbers lead to OPSs. For example, if we set $s_0^*=0$ we get that
\[
\mathcal{L}^{-*}(1)=[1,1]_1=0,
\]
which shows that the corresponding OPS does not exist \tc{and so the case when $s_0^*=0$ should be excluded from our considerations.}

\tc{In order to define a sequence of monic polynomials orthogonal with respect to  $\mathcal{L}^{-*}$, we need to introduce new polynomials $R_n(z)$. Let $\{P_n(z)\}_{n=0}^{\infty}$ be an OPS with respect to a quasi-definite linear $\mathcal{L}$ and let $\{Q_n(z)\}_{n=0}^{\infty}$ be defined by \eqref{def:Qn}. Let $R_n(z)$ be the polynomial of degree $n$ given by
\begin{equation}\label{def:Rn}
R_n(z) = P_n(z)+\frac{1}{s^*_0}Q_n(z)
\end{equation} 
for  $s^*_0 \in \mathbb{R}\setminus \{0\}$. Obviously, $y_n=R_n(z)$ verifies the same difference equation \eqref{DiffEqGeneral} but with a different set of initial data
\[
R_0(z)=1, \quad R_1(z)=z-c_1+\frac{1}{s^*_0}.
\]
Note that if $\mathcal{L}$ is positive-definite then by Favard's theorem $\{R_n(z)\}_{n=0}^\infty$ is an OPS with respect to some positive-definite linear functional as well. Also, if $\{P_n(z)\}_{n=0}^\infty$ is in the Nevai class $\mathcal{N}(a,c)$, then so is $\{R_n(z)\}_{n=0}^\infty$.}

\tc{
Next, one can easily verify that the monic polynomial
\[
P_n^{-*}(\kappa, z)=P_n(z)+A_n P_{n-1}(z),
\]
where 
\[
 A_n=-\frac{s_0^*P_{n}(\kappa)+Q_{n}(\kappa)}{s_0^*P_{n-1}(\kappa)+Q_{n-1}(\kappa)}=-\frac{R_n(\kappa)}{R_{n-1}(\kappa)}
\]
is orthogonal to the monomials $1$, $z$. \dots, $z^{n-1}$ with respect to $\mathcal{L}^{-*}$ provided that $s_0^*P_{n-1}(\kappa)+Q_{n-1}(\kappa)\ne 0$ (for details see \cite{DM14} or \cite{Ger40}).}

The next statement guaranties that under certain conditions the functional $\mathcal{L}^{-*}$ is regular or, which is the same, quasi-definite.\\

\begin{theorem}\label{thm:RegGer}
Let $\mathcal{L}$ be \tc{ a positive-definite linear functional} and let $\{P_n(z)\}_{n=0}^\infty$ be the corresponding monic OPS. \tc{Also, let $\{Q_n(z)\}_{n=0}^{\infty}$ be the polynomials defined by \eqref{def:Qn}.} If $\kappa\in\dC_{\pm}$ and $s_0^*=\mathcal{L}^{-*}(1)\in\overline{\dC}_{\mp}\setminus\{0\}$ then $s_0^*P_{n-1}(\kappa)+Q_{n-1}(\kappa)\ne 0$ for $n=1,2,3,\dots$ 
\tc{ and thus the corresponding Geronimus transformation $\mathcal{L}^{-*}$ at $\kappa$ is a quasi-definite functional}.
\end{theorem}
\begin{proof}
Since the zeros of $Q_{n-1}(z)$ and $P_{n-1}(z)$ interlace, similarly to what was done in the proof of Proposition \ref{prop:zeros} we can conclude that 
\begin{equation}\label{eq:Qn_Pn}
\Im \kappa\Im\left(\dfrac{Q_{n-1}(\kappa)}{P_{n-1}(\kappa)}\right)<0.
\end{equation} Without loss of generality let $\kappa \in \mathbb{C}_+$ and $s^*_0\in \mathbb{C}_-$. Then \eqref{eq:Qn_Pn} implies that $\Im \left(s^*_0+\frac{Q_{n-1}(\kappa)}{P_{n-1}(\kappa)} \right)<0$ hence  $s_0^*P_{n-1}(\kappa)+Q_{n-1}(\kappa)\ne 0$. 
\end{proof}

We are going to also refer to the polynomials $P_n^{-*}(\kappa,z)$'s as the Geronimus transformation of the polynomials $P_n(z)$'s at $\kappa$.

\tc{One of the consequences of Theorem \ref{thm:RegGer} is that the resulting polynomials $\{P_n^{-*}(\kappa,z)\}_{n=0}^\infty$ corresponding to Geronimous transformation satisfies 
\begin{equation}\label{eq:Pn_inverse_Recursive}
zP^{-*}_n(\kappa,z)=P^{-*}_{n+1}(\kappa, z)+c^{-*}_{n+1}P^{-*}_n(\kappa,z)+\lambda_{n+1}^{-*}P^{-*}_{n-1}(\kappa, z). 
\end{equation} Thus, 
 we can see that $\mathcal{L}^{-*}$ corresponds to the following monic Jacobi matrix:
\[
J^{-*}_m=\begin{pmatrix}
c^{-*}_1 & 1 & 0 & \cdots   \\
     \lambda^{-*}_2 & c^{-*}_2 & 1 &   \\
     0 & \lambda^{-*}_3 & c^{-*}_3 & 1 \ddots  \\
     \vdots & & \ddots & \ddots
\end{pmatrix}
\]
 where \begin{equation}\label{lAndcFormulasGT}
\lambda^{-*}_{n+1} = \lambda_n \frac{R_n(\kappa)R_{n-2}(\kappa)}{R^2_{n-1}(\kappa)} ,\quad c^{-*}_{n+1}=c_{n+1}-\frac{R_n(\kappa)}{R_{n-1}(\kappa)}+\frac{R_{n+1}(\kappa)}{R_n(\kappa)}.
\end{equation}
}

\tc{Under the proper conditions the boundedness of the Jacobi matrix is preserved under the Geronimus transformation.}

\begin{proposition}\label{UboundP}
\tc{Let $\{P_n(z)\}_{n=0}^\infty$ be a monic OPS with respect to a positive-definite linear functional and let $J_m$ be the corresponding monic Jacobi matrix. Assume that $J_m$ is bounded, that is, its entries $\lambda_n$ and $c_n$ are bounded and let $\kappa\in\dC_{\pm}$ and $s_0^*\in\overline{\dC}_{\mp}\setminus\{0\}$. Then the corresponding $J^{-*}_m$ is bounded and as a result the set of all zeros of the polynomials $P^{-*}_n(\kappa,z)$'s is bounded.}
\end{proposition}
\begin{proof}
\tc{First, notice that
\begin{equation}\label{FormForRatio}
\frac{R_{n+1}(\kappa)}{R_{n}(\kappa)}=\frac{P_{n+1}(\kappa)}{P_n(\kappa)}
\cdot\frac{s_0^*+\frac{Q_{n+1}(\kappa)}{P_{n+1}(\kappa)}}{s_0^*+\frac{Q_{n}(\kappa)}{P_{n}(\kappa)}}.
\end{equation}
From the Markov theorem (see \cite[Theorem 2.6.2]{Ismail09}) we get that 
\[
\frac{Q_{n}(\kappa)}{P_{n}(\kappa)}\to w,
\]
where $w\in{\dC}_{\mp}$ for $\kappa\in\dC_{\pm}$. Thus, in addition to Theorem \ref{thm:RegGer}, we also know that $s_0^*+w\ne 0$. Hence,
Proposition \ref{prop:zeros} and \eqref{lAndcFormulasGT} yield that $\dfrac{R_n(\kappa)}{R_{n-1}(\kappa)}$ and $\dfrac{R_{n}(\kappa)}{R_{n+1}(\kappa)}$ are bounded sequences.  Then it follows from \eqref{FormForRatio} that the sequences $\lambda_n^{-*}$ and $c_n^{-*}$ are also bounded. Thus, the corresponding Jacobi operator is bounded in $\ell^2$ and as before the boundedness of zeroes follows from \cite[Theorem 3.4 (a)]{B01}. }
\end{proof}

\tc{Just as in the case of the Christoffel transformation of the polynomials $P_n(z)$ at $\kappa\in\dC\setminus\dR$, we have that a Nevai class is invariant under the Geronimus transformation. }
\begin{proposition}
Let $\{P_n(z)\}_{n=0}^\infty$ be a monic OPS with respect to a positive-definite linear functional and let $\kappa\in\dC\setminus\dR$. If $\{P_n(z)\}_{n=0}^\infty$ is in the Nevai class $\mathcal{N}(a,c)$, then so is $\{P^{-*}_n(\kappa,z)\}_{n=0}^\infty$ provided the latter exists. 
\end{proposition}
\begin{proof}  According to the Markov theorem (see \cite[Theorem 2.6.2]{Ismail09}), the sequence $\dfrac{Q_n(z)}{P_n(z)}$ converges uniformly on compact subsets of $\dC\setminus\dR$. \tc{Consequently, it follows from \eqref{FormForRatio} that} $\frac{R_{n+1}(z)}{R_{n}(z)}\rightarrow f(z)$, where $f(z)$ is defined in
\tc{the proof of} Proposition \ref{prop:Nevai_Pn*}. Thus, we see from \eqref{lAndcFormulasGT} that $\lambda^{-*}_n\rightarrow a$ and $c^{-*}_n \rightarrow c$. Hence, $\{P^{-*}_n(\kappa,z)\}_{n=0}^\infty$ is in $\mathcal{N}(a,c)$.
\end{proof}

One can be more specific about locations of zeros for each $n$.
\begin{theorem}\label{prop:*zeroes} 
Let $\{P_n(z)\}_{n=0}^\infty$ be a monic OPS  with respect to a positive-definite linear functional, let $P_n^{-*}(\kappa, z)$ be the Geronimus transformation of $\{P_n(z)\}_{n=0}^\infty$ and let $R_n(z)= P_n(z) + \frac{1}{s^*_0}Q_n(z)$ with $s^*_0\in \mathbb{R}\setminus\{0\}$.  For $n\geq 1$,
\begin{enumerate}[(i)]
    \item If $\kappa \in \mathbb{C}_+, \text{ then the zeros of } P_n^{-*}(\kappa, z) \text{ lie in the horizontal strip }
    \newline \left\{ z\in \mathbb{C}\, \middle\vert\,  0 < \Im z \leq -\frac{1}{\Im\left(\frac{R_{n-1}(\kappa)}{R_n(
\kappa)}\right)}\right\}$.
\item If $\kappa \in \mathbb{C}_-, \text{ then the zeros of } P_n^{-*}(\kappa, z) \text{ lie in the horizontal strip } 
\newline \left\{z\in \mathbb{C} \,\middle\vert \, -\frac{1}{\Im\left(\frac{R_{n-1}(\kappa)}{R_n(
\kappa)}\right)}\leq \Im z <0\right\}.$
\end{enumerate}


\end{theorem}
\proof Let $\kappa \in \mathbb{C}_+$ and suppose $P^{-*}_n(\kappa, z_0)=0$ for some $z_0\in \mathbb{C}_-$ and some $n\geq 1$. Then
\begin{equation}\label{eq:Pn_inverse}P_n(z_0)\left(s^*_0P_{n-1}(\kappa)+Q_{n-1}(\kappa)\right)
= P_{n-1}(z_0)\left(s^*_0P_n(\kappa)+Q_n(\kappa)\right).
\end{equation}  Since $P_n(z)$ has only real zeros for $n=1,2, \dots$, equation (\ref{eq:Pn_inverse}) is equivalent to
\begin{equation}\label{eq:PnRn}\frac{P_{n-1}(z_0)}{P_n(z_0)}=\frac{R_{n-1}(\kappa)}{R_n(\kappa)}.
\end{equation} By Proposition \ref{prop:zeros},
 $\Im \left(\frac{P_{n-1}(z_0)}{P_n(z_0)}\right) >0 $ yet since $\{R_n(z)\}_{n=0}^\infty$ is a monic OPS, $\Im \left(\frac{R_{n-1}(\kappa)}{R_n(\kappa)} \right)<0$ which is a contradiction. Thus the zeros of $P^{-*}_n(\kappa, z)$ must lie in $\mathbb{C}_+$ for all $n=1,2, \dots$.\\
 Now suppose $x_0\in \mathbb{R}$ and there exists $n \in \mathbb{{N}} $ such that $P^{-*}_n(\kappa,x_0)=0$. Then
 \begin{equation}\label{eq:PnRn2}P_n(x_0)=\frac{R_n(\kappa)}{R_{n-1}(\kappa)}P_{n-1}(x_0).
 \end{equation}
 If $x_0$ is not a zero of $P_n(z)$ then (\ref{eq:PnRn}) holds however $\Im \left( \frac{P_{n-1}(x_0)}{P_n(x_0)}\right)=0$ while $\Im\left(\frac{R_{n-1}(\kappa)}{R_n(\kappa)} \right) <0$. If $x_0$ is a zero of $P_n(z)$ then (\ref{eq:PnRn2}) implies $x_0$ must be a zero of $P_{n-1}(z)$ since by Theorem \ref{thm:RegGer}, $s^*_0R_n(\kappa)=s^*_0P_n(\kappa)+Q_n(\kappa) \neq 0$ for any $n=1, 2, \dots$, contradicting the fact that the zeros of $P_n(z)$ and $P_{n-1}(z)$ interlace.

Now if $z_0$ is a zero of $P^{-*}_n(\kappa,z)$ for some $n$ such that
$\Im z_0 > -\frac{1}{\Im\left(\frac{R_{n-1}(\kappa)}{R_n(\kappa)} \right)}$, then 
\[-\frac{1}{\Im z_0}> \Im \left( \frac{R_{n-1}(\kappa)}{R_n(\kappa)}\right)= \Im\left(\frac{P_{n-1}(z_0)}{P_n(z_0)} \right)
\] but by Theorem \ref{thrm: zeros}, $\Im\left( \frac{P_{n-1}(z_0)}{P_n(z_0)}\right)\geq -\frac{1}{\Im z_0}$.
\\
Therefore, the zeros of $P^{-*}_n(\kappa, z)$ can only lie in $\left\{z \in \mathbb{C}: 0 < \Im z \leq -\frac{1}{\Im\left(\frac{R_{n-1}(\kappa)}{R_n(
\kappa)}\right)}\right\}$. This proves $(i)$.\\
 \indent  Since $\overline{\frac{P_{n-1}(z_0)}{P_n(z_0)}}= \frac{P_{n-1}(\overline{z_0})}{P_n(\overline{z_0})}$ and $\overline{\frac{R_{n-1}(z_0)}{R_n(z_0)}}= \frac{R_{n-1}(\overline{z_0})}{R_n(\overline{z_0})}$, $(ii)$ follows from $(i)$.\\
\qed

In the case of a Nevai class, one can get more information about the asymptotic behaviour of zeros. 

\begin{theorem} \label{thrm:zeroes_convergence_geronimous}
Let $\{P_n(z)\}_{n=0}^\infty$ be an OPS in a Nevai class and let $P_n^{-*}(\kappa, z)$ be it Geronimus transformation at $\kappa\in\dC_\pm$ for some $s_0^*\in \mathbb{R}\setminus\{0\}$. Then there exists a \tc{sequence of zeros 
$\{\xi_n \}_{n=1}^\infty$ such that $\xi_n$ is a zero of $P_n^{-*}(\kappa, z)$  and} $\xi_n \rightarrow \kappa$  while the imaginary part of the remaining zeros of the polynomial $P_n^{-*}(\kappa, z)$ converge to zero. 
\end{theorem}

\begin{example}
Consider the monic Chebyshev polynomials $\{{T}_n(x)\}_{n=0}^\infty$ as in Example \ref{Ex:Chebyshev} and put $s_0^*=1$. According to Theorem \ref{prop:*zeroes}, the zeros of 
\[
{T}^{-*}_n(i,z)=T_n(z)+A_nT_{n-1}(z)
\]
lie in $\mathbb{C}_+$ and by Theorem \ref{thrm:zeroes_convergence_geronimous}, they cluster at $i$ as can be seen in Figures \ref{fig:ChebyshevInverse1} and \ref{fig:ChebyshevInverse2}.
\end{example}

\begin{figure}[ht]
\centering
\begin{minipage}{.5\textwidth}
  \centering
  \includegraphics[scale=0.25]{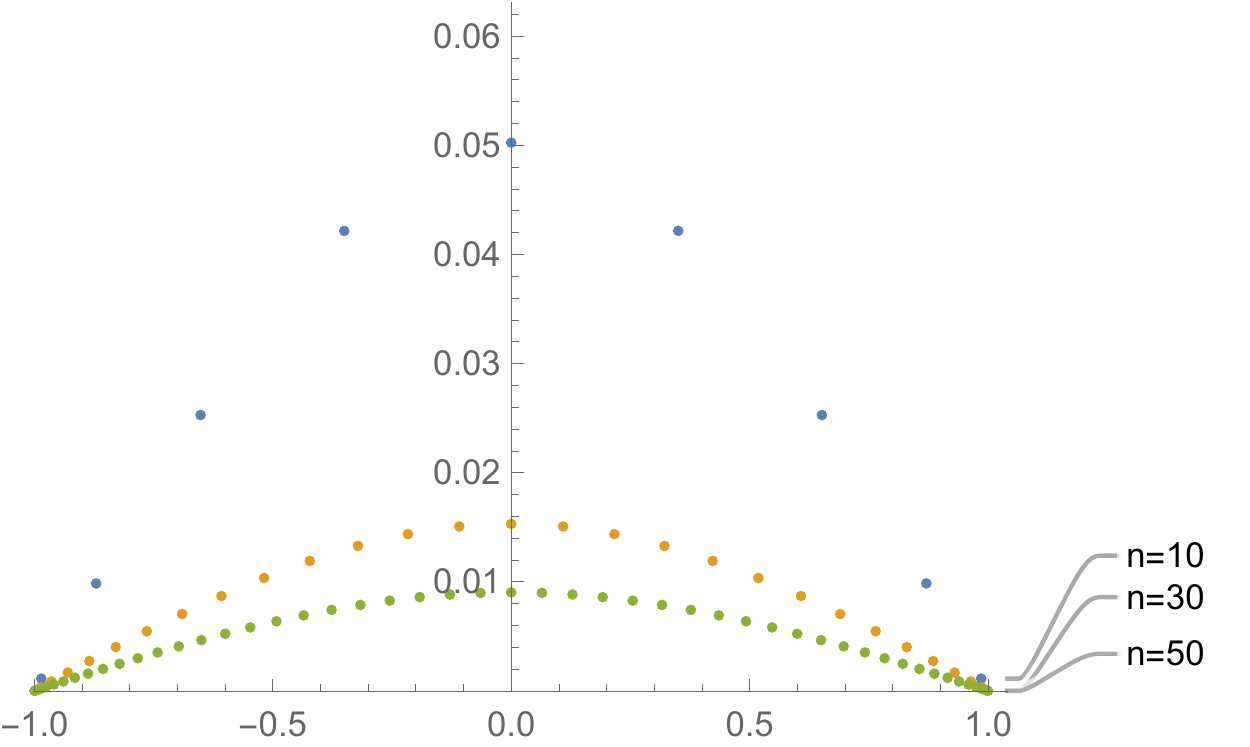}
  \caption{The behavior of the zeroes of ${T}_n^{-*}(z)$ for $s_0^*=1$ and $\kappa=i$ when $n$ is increasing.}
    \label{fig:ChebyshevInverse1}
\end{minipage}%
\begin{minipage}{.5\textwidth}
  \centering
  \includegraphics[width=\linewidth]{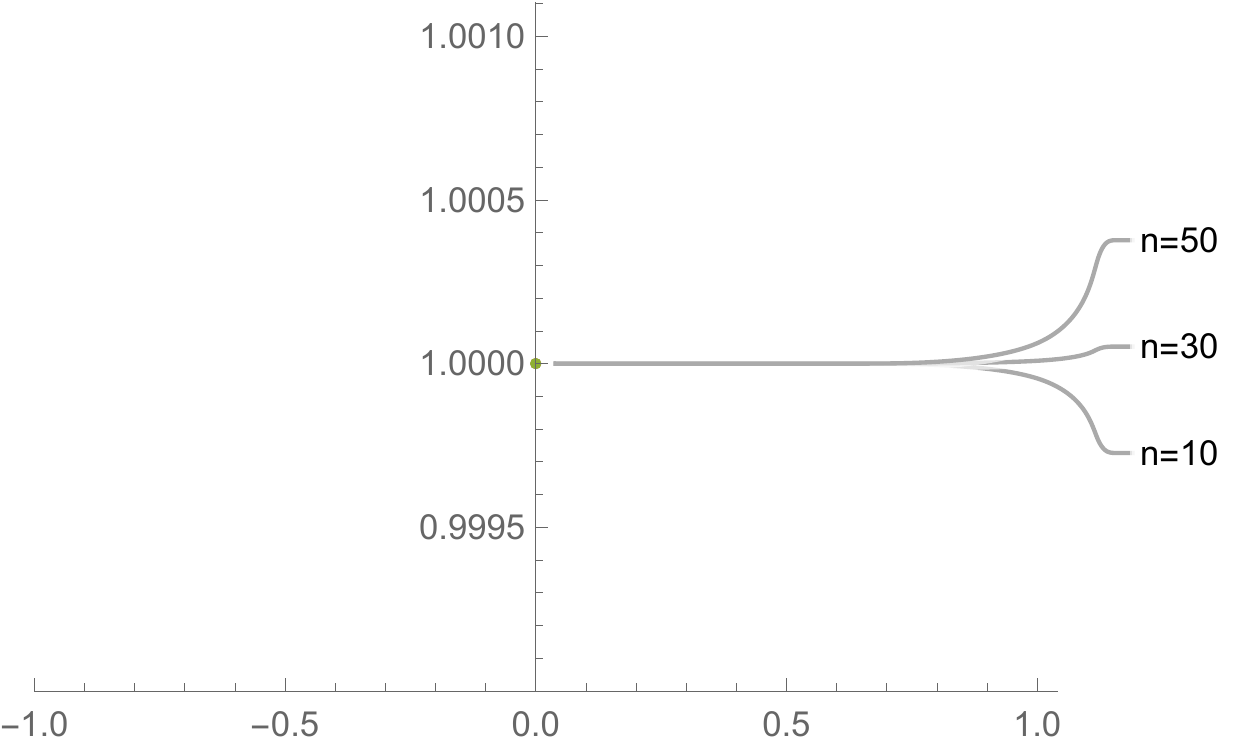}
  \caption{The behavior of the zeroes of ${T}_n^{-*}(z)$ at the neighborhood of $i$ for $s_0^*=1$ and $\kappa=i$ when $n\to\infty$.}
    \label{fig:ChebyshevInverse2}
\end{minipage}
\end{figure}
A similar behavior takes places if $\kappa=1+i$.
\begin{figure}[h!]
\centering
\begin{minipage}{.5\textwidth}
  \centering
  \includegraphics[width=\linewidth]{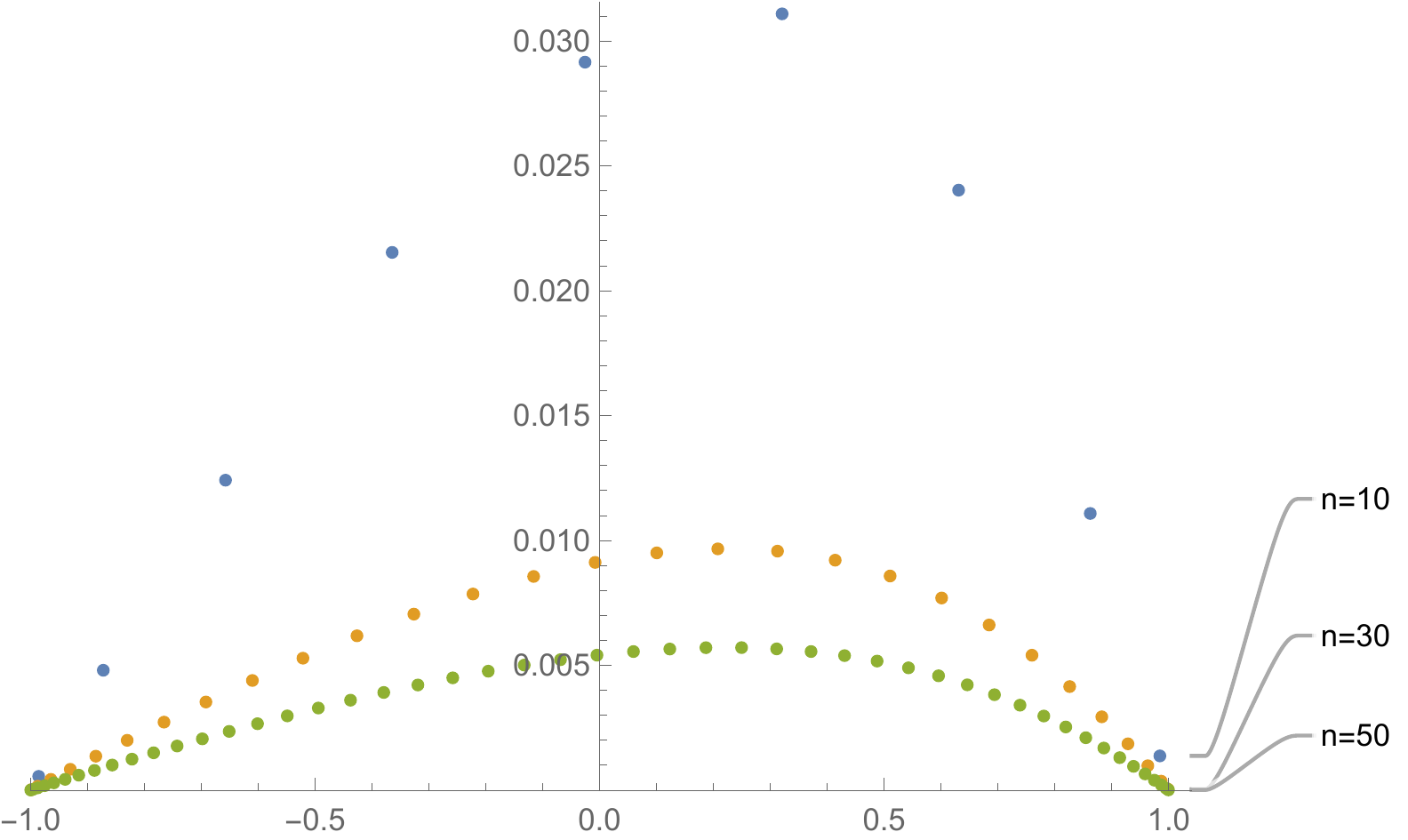}
  \caption{The behavior of the zeroes of ${T}_n^{-*}(z)$ for $s_0^*=1$ and $\kappa=1+i$ when $n$ is increasing.}
    \label{fig:ChebyshevInverse_1+i}
\end{minipage}%
\begin{minipage}{.5\textwidth}
  \centering
  \includegraphics[ scale=0.4]{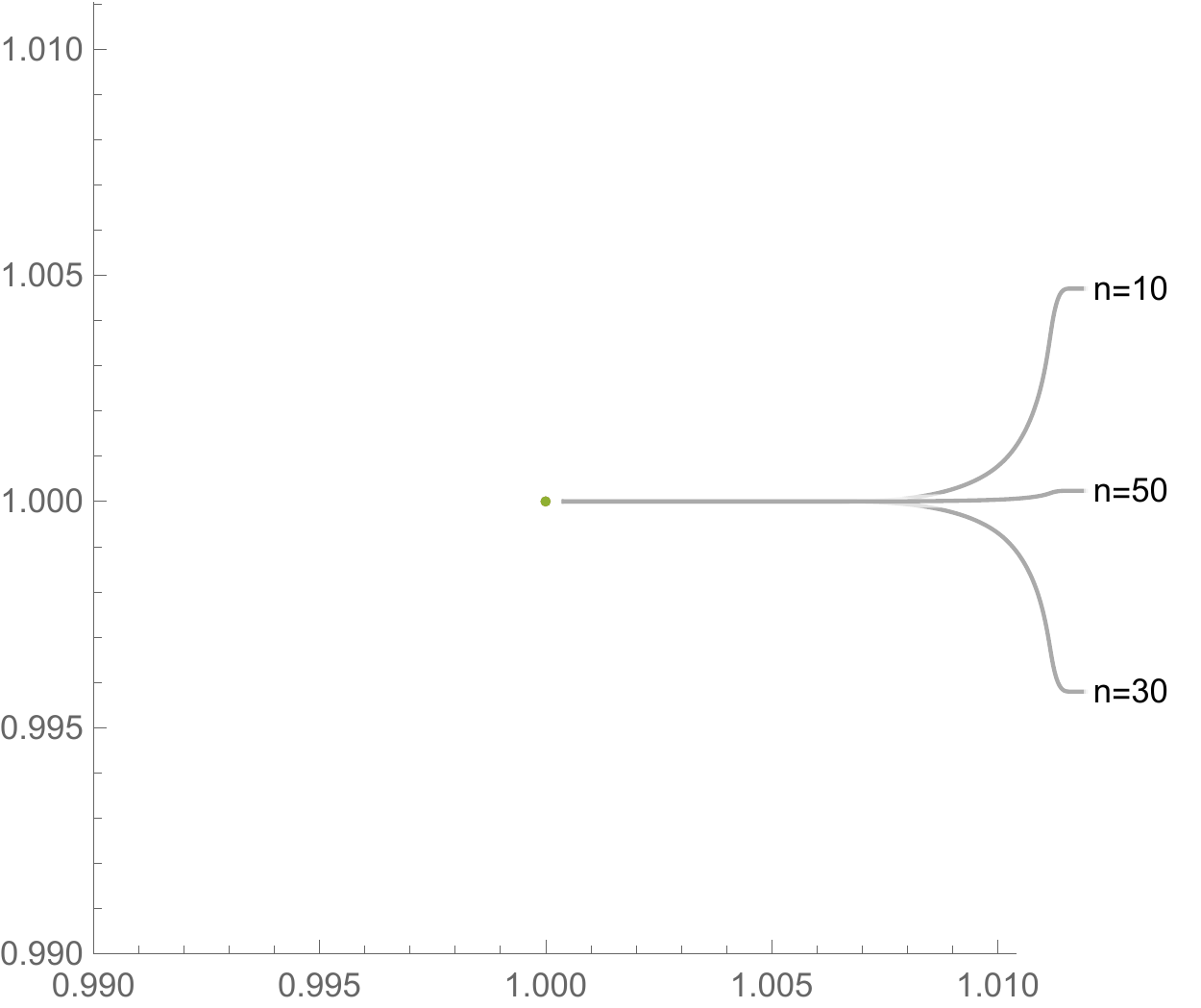}
  \caption{The behavior of the zeroes of ${T}_n^{-*}(z)$ at the neighborhood of $1+i$ for $s_0^*=1$ and $\kappa=1+i$ when $n\to\infty$.}
    \label{fig:ChebyshevInverse_zoom
    1+i}
\end{minipage}
\end{figure}

\proof Without loss of generality let $\kappa \in \mathbb{C}_+$. Notice that $P_n^{-*}(\kappa, z)$ can be re-written as
\begin{equation*}
   P_{n}^{-*}(\kappa, z)= P_n(z)-\left(\frac{R_n(\kappa)}{R_{n-1}(\kappa)}\right)P_{n-1}(z)
\end{equation*} where $R_n(z)=P_n(z)+\frac{1}{s^*_0}Q_n(z)$ as in $(\ref{def:Rn})$.
Let $H_n(z)=\frac{P_n(z)}{P_{n-1}(z)}-\frac{R_n(z)}{R_{n-1}(z)}.$ Then $H_n(z)$ is defined on $\mathbb{C}\setminus \mathbb{R}$ and $H_n(z)$ and $P^{-*}_n(\kappa,z)$ have the same zeros. Now since $\{P_n(z)\}_{n=0}^\infty$ and $\{R_n(z)\}_{n=0}^\infty$ are in the same Nevai class, we see that $H_n(z)$ converges uniformly to $f(z)-f(\kappa)$ on compact subsets of $\mathbb{C}_+$ where $f(z)= \frac{(z-c)+\sqrt{(z-c)^2-4a^2}}{2}$. Let $r>0$ and let $\Delta_r=\{z\in \mathbb{C}: |z-\kappa|<r\}.$ Then since $f(z)-f(\kappa) \not\equiv 0$, $\kappa$ is an isolated zero of $f(z)-f(\kappa)$ and thus $f(z)-f(\kappa)$ has no zeros in $\Delta_r \setminus \{\kappa\}$ for any $r$ sufficiently small. Just as in Theorem \ref{thrm: n-zero behavior}, we have by Hurwitz's Theorem that $H_n(z)$ has only a simple zero in $\Delta_r$ for large $n$. Since $r$ can be taken arbitrarily small, we see that there is a subsequence $\{{z}_{n_k}\}_{k=0}^\infty = \{\xi_{n_k}\}_{k=0}^\infty$ of zeros of $\{H_n(z)\}_{n=0}^\infty$ (and hence of $\{P^{-*}_n(\kappa,z)\}_{n=0}^\infty$) converging to $\kappa$.

Let $\hat{z}_n$ be such that $\Im \hat{z}_n= \max\{ \Im z_{n,j}: z_{n,j} \text{ is a zero of }P_n^{-*}(\kappa, z),\,\, {z}_{n,j}\neq {\xi}_{n_k}\}$. Then there exists a convergent subsequence $\{\hat{z}_{n_k}\}_{k=0}^\infty$ with limit $z_0 \in \mathbb{C}$. As before, if $z_0\in \mathbb{C}_+$ then $H_n(\hat{z}_{n_k})\rightarrow f(z_0)-f(\kappa)$ so by the injectivity of $f(z)$, it must be the case that $z_0=\kappa$. Fixing $r>0$, we know by Hurwitz's Theorem that $P_n^{-*}(\kappa, z)$ has only a simple zero in $\Delta_r$ which is $\xi_{n_k}$. Since $P_n^{-*}(\kappa, z)$ has no other zeros in $\Delta_r$ for large $n$, $\hat{z}_{n_k}$ cannot converge to $\kappa$. Thus, since $\hat{z}_{n_k}\in \mathbb{C}_+$, it follows that $\Im z^*_{n_k}\rightarrow 0$. By definition of $\hat{z}_n$, this shows that the imaginary part of the remaining zeros of $\{P^{-*}_n(\kappa, z)\}_{n=0}^\infty$ must converge to 0.
\qed

\begin{remark}
\tc{Theorem \ref{thrm:zeroes_convergence_geronimous} suggests that the Jacobi matrix corresponding to the Geronimus transformation at $\kappa$ should have $\kappa$ as an eigenvalue and it will be shown later in Section 5 that the latter statement holds true in general and it is not specific for Nevai classes. Also,} note that in Figure \ref{fig:ChebyshevInverse2}, there is a cluster point of zeros at $z=i$. However, the convergence to $i$ occurs at too quick of a rate for Mathematica to distinguish. Thus, it makes sense to consider a different scale. In particular, Figure \ref{fig:log_graph} suggests that $\xi_n \rightarrow \kappa$ at an exponential rate. Besides, apparently, an estimate for $|\xi_n-\kappa|$ similar to \cite[Theorem 4.3]{DPW16} holds  in this case too. 

\begin{figure}[h!]
    \centering
    \includegraphics[width=0.5\textwidth]{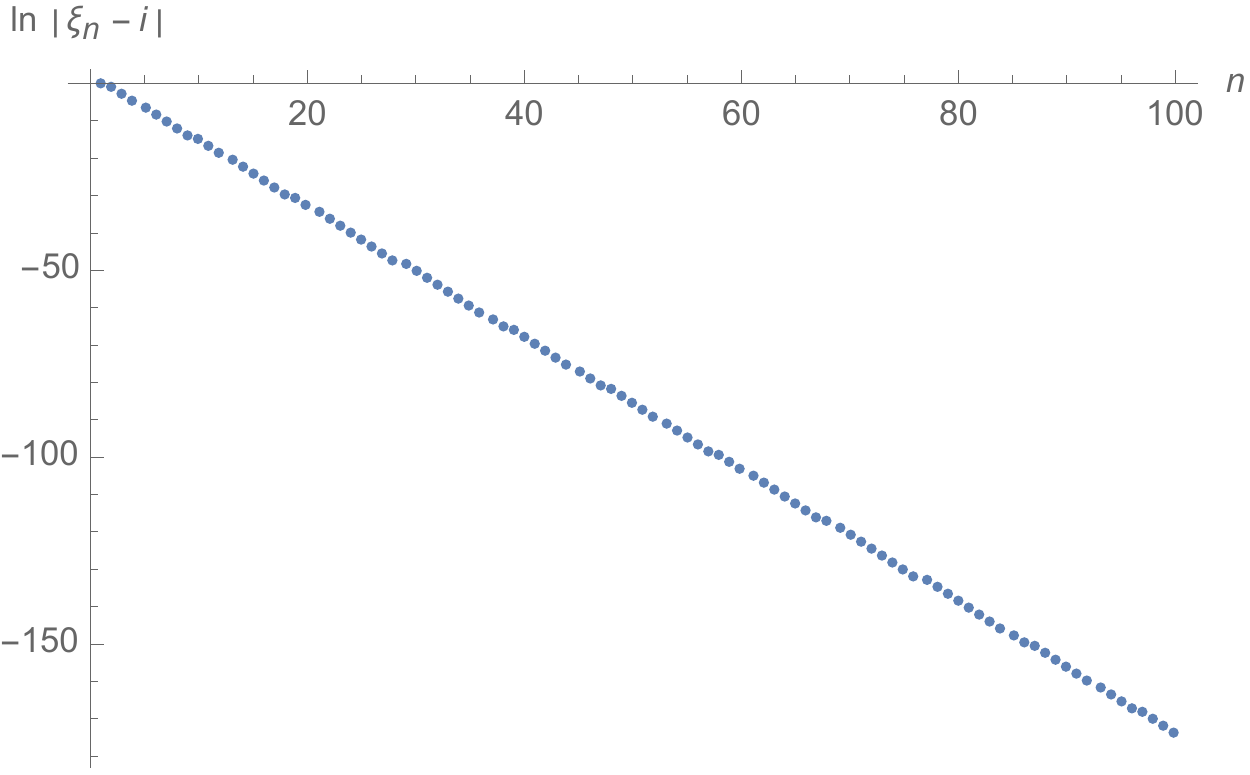}
    \caption{The graph of $\ln|\xi_n-i|$
     for $n=1,2, \dots, 100$.}
    \label{fig:log_graph}
\end{figure}
\end{remark}

For the Geronimus transformation at $\kappa\in\dC\setminus\dR$, the ratio asymptotic is preserved at all points except for $\kappa$ \tc{and we will see later that $\kappa$ is in the spectrum of the underlying Jacobi matrix.} 

\begin{theorem}
Let $\{P_n(z)\}_{n=0}^\infty$ be a monic OPS with respect to a \tc{positive-definite} linear functional $\mathcal{L}$ and let $\{P_n^{-*}(\kappa,z)\}_{n=0}^\infty$ be the corresponding OPS for the Geronimus transformation $\mathcal{L}^{-*}$. Let $f(z)$ be defined as in the proof of Theorem \ref{thrm: n-zero behavior}. If $\{P_n(z)\}_{n=0}^\infty$ is in a Nevai class then $\frac{P_{n}^{-*}(\kappa, z)}{P_{n-1}^{-*}(\kappa,z)}$ converges uniformly to $f(z)$ on compact subsets of $\mathbb{C}\setminus(\mathbb{R}\cup \{\kappa\})$.
\end{theorem}
\proof Without loss of generality, let $\kappa \in \mathbb{C}_+$ and let $H_n(z) =\frac{P_n(z)}{P_{n-1}(z)}-\frac{R_n(\kappa)}{R_{n-1}(\kappa)}$ as in the proof of Theorem \ref{thrm:zeroes_convergence_geronimous}. Then 
\[\frac{P_{n}^{-*}(\kappa,z)}{P_{n-1}^{-*}(\kappa,z)}=\frac{P_{n-1}(z)}{P_{n-2}(z)}\frac{H_{n}(z)}{H_{n-1}(z)}.
\]
Recall that $H_n(z)$ and $P_n^{-*}(\kappa,z)$ have the same zeros thus by Theorem \ref{thrm:zeroes_convergence_geronimous}, $H_n(z)$ is non-zero on compact subsets of $\mathbb{C}\setminus \mathbb{R}$ not containing $\kappa$ by taking $n$ large enough. Therefore, $\frac{H_{n}(z)}{H_{n-1}(z)}$ is holomorphic on compact subsets of $\mathbb{C}\setminus(\mathbb{R}\cup\{\kappa\})$ for large $n$. Since $\{P_n(z)\}_{n=0}^\infty$ is in a Nevai class, $H_n(z)$ converges uniformly to $f(z)-f(\kappa)$, therefore, $\frac{H_{n}(z)}{H_{n-1}(z)}\rightarrow 1$ uniformly on compact subsets of $\mathbb{C}\setminus(\mathbb{R} \cup \{\kappa\})$ and as a result,  $\frac{P_{n}^{-*}(\kappa,z)}{P_{n-1}^{-*}(\kappa,z)}\rightarrow f(z)$ uniformly on compact subsets of $\mathbb{C}\setminus(\mathbb{R} \cup \{\kappa\})$ .\\
\qed

\tc{In principle, as in the case of Christoffel transformation, one can iterate Geronimus transformations. Say,  if $\kappa_1\in\dC_{+}$ and $s_0^*=\mathcal{L}^{-*}(1)\in\overline{\dC}_{-}\setminus\{0\}$, then according to Theorem \ref{thm:RegGer} the OPS $\{P_n^{-*}(\kappa, z)\}_{n=0}^\infty$, where 
\begin{equation*}
 P_n^{-*}(\kappa,z)=\frac{1}{R_{n-1}(\kappa)}
 \begin{vmatrix}
 R_{n-1}(\kappa) &R_n(\kappa)\\
 P_{n-1}(z) &P_n(z)
 \end{vmatrix},
\end{equation*}
and $R_n(z)$ is defined by equation $(\ref{def:Rn})$, exists}. If, in addition, we have a complex number $\kappa_2$ and another complex number $s_0^{**}=\mathcal{L}^{-**}(1)$ such that $s_0^{**}P_{n-1}^{-*}(\kappa_1,\kappa_2)+Q_{n-1}^{-*}(\kappa_1,\kappa_2)\ne 0$ for $n=1,2,3,\dots$, we see that the polynomials
\begin{equation}
\resizebox{\textwidth}{!}{$
P^{-**}_n(\kappa_1,\kappa_2, z) = \frac{
 \begin{vmatrix}
 s_0^{**}P_{n-1}^{-*}(\kappa_1,\kappa_2)+Q_{n-1}^{-*}(\kappa_1,\kappa_2) &s_0^{**}P_n^{-*}(\kappa_1,\kappa_2)+Q_n^{-*}(\kappa_1,\kappa_2)\\
 P_{n-1}^{-*}(\kappa_1,z) &P_n^{-*}(\kappa_1, z)
 \end{vmatrix}}
 {s_0^{**}P_{n-1}^{-*}(\kappa_1,\kappa_2)+Q_{n-1}^{-*}(\kappa_1,\kappa_2)}
 $}
\end{equation}
are correctly defined for any nonnegative integer $n$ and so they are orthogonal with respect to $\mathcal{L}^{-**}=(\mathcal{L}^{-*})^{-*}$.
\tc{To conclude we are going to formulate a statement about iterations of Geronimus transformations corresponding to the evolution
\[
d\mu(t)\to\frac{d\mu(t)}{t-\kappa}\to\frac{d\mu(t)}{|t-\kappa|^2}
\]
and it will be used in Section 6.}
\begin{theorem}\label{IteratedGeronimus}\tc{
Given a linear functional of the form \eqref{functional:IntRepr}, let $\kappa\in\dC_{\pm}$ and $\displaystyle{s_0^*=\int_a^b\dfrac{d\mu(t)}{t-\kappa}}$. Then the corresponding polynomials $P_n^{-*}(\kappa, z)$ are correctly defined for each $n$ and they are orthogonal with respect to the complex-valued measure $d\mu(t)/(t-\kappa)$. Also, if we further set $\displaystyle{s_0^{**}=\int_{a}^{b}\frac{d\mu(t)}{|t-\kappa|^2}}$, the resulting iterated polynomials $P_n^{-**}(\kappa, \overline{\kappa}, z)$ are correctly defined for each $n$ and they are orthogonal with respect to the positive measure $d\mu(t)/|t-\kappa|^2$.
}
\end{theorem}
\begin{remark}
\tc{
It should be emphasized here that if $\kappa\in\dC_{\pm}$ then $\displaystyle{s_0^*=\int_a^b\dfrac{d\mu(t)}{t-\kappa}}\in\dC_{\pm}$ and therefore the statement about $P_n^{-*}(\kappa, z)$ complements Theorem \ref{thm:RegGer}.
}
\end{remark}
\begin{proof}
\tc{
By contradiction, assume that $P_n^{-*}(\kappa, z)$ cannot be defined for some $n$, which is equivalent to 
\begin{equation*}
    s_0^*P_{n-1}(\kappa)+Q_{n-1}(\kappa)=0,
\end{equation*}
which using \eqref{def:Qn} can be rewritten as
\[
\int_a^bP_{n-1}(t)\dfrac{d\mu(t)}{t-\kappa}=0.
\] 
Furthermore, since $P_{n-1}$ is orthogonal with respect to $d\mu$, by induction we get that
\[
\int_a^bP_{n-1}(t)t^k\dfrac{d\mu(t)}{t-\kappa}=0, \quad k=0,1,\dots, n-1.
\]
Therefore,
\[
\int_a^bP_{n-1}^2(t)\dfrac{d\mu(t)}{t-\kappa}=0
\]
or taking the imaginary part
\[
\int_a^bP_{n-1}^2(t)\dfrac{d\mu(t)}{|t-\kappa|^2}=0,
\]
which is impossible. As a result, $P_n^{-*}(\kappa, z)$ is defined for each $n$ and according to \eqref{eq:bilinear} the orthogonality measure is  $d\mu(t)/(t-\kappa)$. Finally, if we then set $\displaystyle{s_0^{**}=\int_{a}^{b}\frac{d\mu(t)}{|t-\kappa|^2}}$, \eqref{eq:bilinear} shows that the measure is $d\mu(t)/|t-\kappa|^2$ and thus the polynomials $P_n^{-**}(\kappa, \overline{\kappa}, z)$ are correctly defined since the measure is finite and positive-definite. 
}
\end{proof}
It is worth pointing out that one can also consider so-called multiple Geronimus transformations that lead to more general orthogonal systems \cite{DM14}, \cite{DGM14} such as Sobolev orthogonal polynomials and derive similar results.

\section{Symmetric Jacobi matrices}

Here we analyze the underlying complex symmetric Jacobi matrices and their spectra. In particular, we show that the Christoffel transformation at $\kappa$ is isospectral and that the Geronimus transformation at $\kappa$ adds the nonreal $\kappa$ to the spectrum.

At first, for the convenience of the reader, recall that for an OPS $\{P_n(z)\}_{n=0}^\infty$, the normalized polynomials
\[
\hat{P}_n(z)=\frac{P_n(z)}{\sqrt{\lambda_1}\sqrt{\lambda_2}\dots\sqrt{\lambda_{n+1}}},
\]
\tc{where one can, in fact, take either of the two values of the square root of the complex number ${\lambda_j}$,}
satisfy the relation 
\begin{equation}\label{SymTriTermNSec5}
a_k\hat{P}_{k+1}(z)+b_k\hat{P}_k(z)+a_{k-1}\hat{P}_{k-1}(z)=z\hat{P}_k(z), \quad k=1,2,3\dots,
\end{equation}
with the initial conditions
\begin{equation}\label{InConNSec5}
\hat{P}_0(z)=1,\quad \hat{P}_1(z)=(z-b_0)/a_1.
\end{equation}
Then the symmetric Jacobi matrix
\begin{equation}\label{JacOpSec5}
J=\begin{pmatrix}
b_0 & a_0 & 0 & \cdots   \\
     a_0 & b_1 & a_1 &   \\
     0 & a_1 & b_2 & \ddots  \\
     \vdots & & \ddots & \ddots
\end{pmatrix}
\end{equation}
is the matrix representation of the operator of the multiplication by $z$
\begin{equation}\label{MatrixFormSec5}
Jp(z)=zp(z),
\end{equation}
where $p(z)=(\hat{P}_0(z), \hat{P}_1(z), \hat{P}_2(z),\dots)^{\top}$. We will say that $J$ corresponds to the underlying \tc{quasi-definite linear }functional $\mathcal{L}$. In the standard way, such a Jacobi matrix generates a closed linear operator acting in the Hilbert space $\ell^2$ and we will also denote this operator by $J$. \tc{Note that this operator $J$ is Hermitian in case $\mathcal{L}$ is positive-definite and otherwise is non-Hermitian.} Next, let us consider the following form of the $LU$-factorization of the tridiagonal matrix $J-\kappa I$ 
\begin{equation}\label{symLU}
J-\kappa I = \sL(\kappa) D(\kappa) \sL^{\top}(\kappa),
\end{equation}
where $D(\kappa)=\diag(d_0(\kappa),d_1(\kappa),....)$ is a diagonal matrix and $\sL\tc{(\kappa)}$ is a lower bidiagonal matrix 
\[      
\sL\tc{(\kappa)} =\begin{pmatrix}
1 & 0 & 0 & \cdots   \\
     v_0\tc{(\kappa)} & 1 & 0 &   \\
     0 & v_1\tc{(\kappa)} & 1 & \ddots  \\
     \vdots & & \ddots & \ddots
\end{pmatrix}
\] \tc{ whose entries depend on $\kappa$}. 
Comparing the entries of \eqref{symLU} gives
\begin{equation}\label{LU_entries1}
d_0(\kappa)=b_0-\kappa, \quad    
      d_j(\kappa)v_j(\kappa) = a_{j} , \quad
      d_{j+1}(\kappa)=b_{j+1}-\kappa-d_j(\kappa)v_j^2(\kappa).
\end{equation}
\tc{Writing} \eqref{SymTriTermNSec5} in the form
\[
a_j\frac{\hat{P}_{j+1}(\kappa)}{\hat{P}_j(\kappa)}=\kappa-b_j-a_{j-1}
\frac{\hat{P}_{j-1}(\kappa)}{\hat{P}_{j}(\kappa)},
\]
we get 
\begin{equation}\label{LU_entries}
d_j(\kappa)=-a_j\frac{\hat{P}_{j+1}(\kappa)}{\hat{P}_j(\kappa)},\quad 
v_j(\kappa) =-
\frac{\hat{P}_{j}(\kappa)}{\hat{P}_{j+1}(\kappa)}.
\end{equation}
In fact, it is now more convenient to have \eqref{symLU} as follows
\begin{equation}\label{symLU_C}
J-\kappa I = L(\kappa) L^{\top}(\kappa)
\end{equation}
with
\begin{equation}\label{DefC}
L=\sL D^{1/2}=\begin{pmatrix}
\sqrt{d_0} & 0 & 0 & \cdots   \\
     v_0\sqrt{d_0} & \sqrt{d_1} & 0 &   \\
     0 & v_1\sqrt{d_1} & \sqrt{d_2} & \ddots  \\
     \vdots & & \ddots & \ddots
\end{pmatrix},
\end{equation}
where one can actually take either of the two values of the square root of the complex number ${d_j}$. \tc{In order to relate this construction to the Christoffel transformation discussed before, we are going to formulate and to prove a standard result (e.g. see \cite{BM04}) in the explicit form where we emphasize the existence condition.} 
\begin{proposition}\label{symDarbouxTr}
Let $J$ be a Jacobi matrix corresponding to a positive-definite \tc{linear} functional $\mathcal{L}$ and let $\kappa\in\dC_\pm$. Then the factorization
\begin{equation}\label{DarFac}
J-\kappa I=LL^{\top}
\end{equation}
exists and the tridiagonal matrix 
\[
{J}_{C}(\kappa)=L^{\top}L+\kappa I
\]
corresponds to the quasi-definite functional $\mathcal{L}^*$ \tc{and thus $J_C(\kappa)$ is a symmetrization of the monic Jacobi matrix $J^*_m(\kappa)$.} 
\end{proposition}
\begin{proof}
From \eqref{LU_entries} we conclude that the factorization exists if and only if $\hat{P}_j(\kappa)\ne 0$ for all admissible $j$, which is true since the polynomials $\hat{P}_j(z)$'s correspond to a positive-definite functional and hence they have only real zeros. Next, using \eqref{MatrixFormSec5} one can see that 
\[
\begin{split}
(J-\kappa I)p(t)=(t-\kappa)p(t)\Rightarrow LL^{\top}p(t)=(t-\kappa)p(t)\Rightarrow\\
\Rightarrow (L^{\top}L)L^{\top}p(t)=(t-\kappa)L^{\top}p(t)\Rightarrow 
{J}_C(\kappa)L^{\top}p(t)=tL^{\top}p(t).
\end{split}
\]
The latter relation suggests that $L^{\top}p(t)$ should be a vector of the orthogonal polynomials corresponding to ${J}_C(\kappa)$. However, the entries of $L^{\top}p(t)$ vanish at $t=\kappa$ and so $L^{\top}p(t)$ doesn't satisfy the proper initial condition, which is $\widetilde{P}_0=1$.
Nevertheless, introducing
\[
\widetilde{p}(t)=\frac{1}{t-\kappa}L^{\top}p(t)=(\widetilde{P}_0(t), \widetilde{P}_1(t), \dots)^{\top}
\]
resolves the issue and it leads to the polynomials
\[
\widetilde{P}_j(t)=\sqrt{d_j(\kappa)}\frac{\hat{P}_j(t)-\frac{\hat{P}_j(\kappa)}{\hat{P}_{j+1}(\kappa)}\hat{P}_{j+1}(t)}{t-\kappa}
\]
which are proportional to $P_n^*(\kappa,z)$ and so they are orthogonal with respect to $\mathcal{L}^*$.
\end{proof}
Evidently, the matrix $L$ also defines a closed linear operator on $\ell^2$ and it turns out that this operator, which will be denoted by $L$ as well, is bounded.
\begin{proposition}\label{Lbound}
Let $J$ be a Jacobi matrix corresponding to a positive-definite \tc{linear} functional $\mathcal{L}$. If $J$ is a bounded operator, that is, the two sequences $a_n$ and $b_n$ are bounded, then $L$ is a bounded operator. 
\end{proposition}
\begin{proof}
Since $L$ is a banded matrix, that is, it has only two nonzero diagonals, it suffices to show that the entries of $L$ are bounded. To this end, note that the nonzero elements of $L$ are expressed in terms of the two sequences
\[
d_j(\kappa)=-a_j\frac{\hat{P}_{j+1}(\kappa)}{\hat{P}_j(\kappa)}=-\frac{{P}_{j+1}(\kappa)}{{P}_j(\kappa)},\quad
v_j(\kappa) =-\frac{\hat{P}_{j}(\kappa)}{\hat{P}_{j+1}(\kappa)}=-a_j\frac{{P}_{j}(\kappa)}{{P}_{j+1}(\kappa)}.
\]
As a result, the nonzero entries of $L$ are bounded due to \eqref{DefC} and statement (iii) of Proposition \ref{prop:zeros}.
\end{proof}
\begin{remark}
It is worth mentioning here that for some real $\kappa$ the decomposition \eqref{DarFac} exists but the corresponding $L$ is unbounded (see \cite{D15})
\end{remark}

To give another flavor to Proposition \ref{symDarbouxTr}, recall that the $m$-function or Weyl function of the Jacobi operator $J$ is the function
\begin{equation*}
 m(J;z)= ((J-zI)^{-1}e_0,e_0)_{\ell^2},
\end{equation*}
where $e_0=(1,0,\dots,0,\dots)^\top\in\ell^2$. It is well known and is not so hard to see that the function $m(J;z)$ is holomorphic on the resolvent set $\rho(J)$ of the operator $J$. Moreover, if $J$ is bounded we have
\begin{equation}\label{mAsympt}
 m(J;z)=-\sum_{j=0}^{\infty}\frac{(J^je_0,e_0)_{\ell^2}}{z^{j+1}},\quad (|z|>\Vert J\Vert).
\end{equation}
Note that if the underlying functional $\mathcal{L}$ is normalized in the way that $\mathcal{L}(1)=1$ then 
\[
\mathcal{L}(z^j)=(J^je_0,e_0)_{\ell^2}, \quad  j=0,1,2, \dots
\]
(see \cite{B01} for more details).
Since the functionals $\mathcal{L}$ and $\mathcal{L}^*$ are related, there is a simple formula that also relates the $m$-functions of $J$ and $J_C(\kappa)$, \tc{which is a way to show that $J_C(\kappa)$ corresponds to $(x-\kappa)d\mu(x)$, where $d\mu$ generates $J$, based on the defintion of $J_C(\kappa)$ given in this section.}
\begin{proposition}\label{prop:JandJc}
Let $J$ be a Jacobi matrix corresponding to a positive-definite \tc{linear} functional $\mathcal{L}$. If $J$ is bounded then 
\begin{equation}\label{formula:JandJc}
m(J_C(\kappa);z)=\frac{1}{b_0-\kappa}\left[(z-\kappa)m(J;z)+1\right] 
\end{equation}
for all complex numbers $z$, whose absolute value $|z|$ is sufficiently large.
\end{proposition}
\begin{proof}
It is worth starting by stressing that if $J$ is bounded then Proposition \ref{prop:zeros} and formulas \eqref{lAndcFormulasCT} give the boundedness of $J_C(\kappa)$. Thus, $m(J_C(\kappa);z)$ and $m(J;z)$ are holomorphic in a neighborhood of $\infty$. Now, we are to prove that the asymptotic expansions of the right-hand and left-hand sides of \eqref{formula:JandJc} coincide and this yields the desired result. In other words, we need to show 
\begin{equation}\label{MomRel}
  \Big((J^{j+1}-\kappa J^j)e_0,e_0\Big)_{\ell^2}=
(b_0-\kappa)(J_C(\kappa)^je_0,e_0)_{\ell^2}, \quad j=0,1,2, \dots.  
\end{equation}
To this end, we are going to prove by induction that 
\[
J^{j+1}-\kappa J^j=LJ_C(\kappa)^jL^\top, \quad j=0,1,2, \dots,
\]
which is evident when $j=0$. Next, assuming the relation is true for $j=k$, we get
\[
\begin{split}
  J^{k+2}-\kappa J^{k+1}&=J(J^{k+1}-\kappa J^{k})=JLJ_C(\kappa)^jL^\top=
(LL^\top+\kappa I)LJ_C(\kappa)^kL^\top\\
&=L(L^\top L+\kappa I)J_C(\kappa)^kL^\top=LJ_C(\kappa)^{k+1}L^\top,  
\end{split}
\]
which shows the validity of \eqref{MomRel}.
Finally, it remains to observe that
\[
(LJ_C(\kappa)^{j}L^\top e_0,e_0)_{\ell^2}=(J_C(\kappa)^{j}L^\top e_0,\overline{L^\top e_0})_{\ell^2}=(b_0-\kappa)(J_C(\kappa)^{j} e_0,e_0)_{\ell^2}
\]
since $L^\top e_0=\sqrt{d_0}e_0$.
\end{proof}

The above observations allow us to prove that the Christoffel transformation at $\kappa$ preserves the spectrum $\sigma(J)$ of the underlying Jacobi operator $J$.
\begin{theorem}\label{Isospectral}
If $J$ is a bounded operator then for any $\kappa\in\dC_\pm$ we have
\[
\sigma(J)=\sigma({J}_C(\kappa)).
\]
\end{theorem}
\begin{proof}
It is well known that for bounded operators $A$ and $B$, we have
\[
\sigma(AB)\setminus\{0\}=\sigma(BA)\setminus\{0\}
\]
(for example see \cite{Barnes98}). This fact, Proposition \ref{Lbound}, and the fact that $\sigma(J)$ cannot have nonreal numbers in it immediately give
\[
\sigma({J}_C(\kappa))\setminus\{\kappa\}=\sigma(J).
\]
Finally, by way of contradiction, assume that $\kappa\in\sigma({J}_C(\kappa))$. Then, it follows from \cite[Theorem 2.14]{B01} that $m(J_C(\kappa);z)$ has a pole at $\kappa$, which is impossible due to \eqref{formula:JandJc}. Thus,  $\kappa\not\in\sigma({J}_C(\kappa))$.
\end{proof}
\begin{example}
\tc{ Let us consider the symmetrization of the monic Jacobi matrix introduced in Example \ref{FibJac}, that is, the symmetric complex Jacobi matrix $J$, whose entries are
\[
a_{n}=\frac{1}{2}\sqrt{\frac{(-1)^{n}}{F_{n+1}^2}+1}, \quad
b_{n}=i\frac{(-1)^n}{2F_{n}F_{n+1}}, \quad n=0, 1, 2, \dots.
\]
Theorem \ref{Isospectral} then implies that $\sigma(J)=[-1,1]$. Note that in this case the Jacobi matrix is in the Nevai class and using the usual spectral tool, Weyl's theorem, one can only conclude that $\sigma_{ess}(J)=[-1,1]$.
}
\end{example}
Now, let us find a $UL$-factorization of the tridiagonal matrix $J-\kappa I$
\begin{equation}\label{symUL}
J-\kappa I = \cU(\kappa) \cD(\kappa) \cU^{\top}(\kappa),
\end{equation}
where $\cD(\kappa)=\diag({t}_0(\kappa),t_1(\kappa),....)$ is a diagonal matrix and $\cU\tc{(\kappa)}$ is an upper bidiagonal matrix 
\[      
\cU\tc{(\kappa)}=\begin{pmatrix}
u_0\tc{(\kappa)}& 1 & 0 & \cdots   \\
0   & u_1\tc{(\kappa)} & 1 &   \\
 0 & 0 & u_2\tc{(\kappa)} & \ddots  \\
     \vdots & & \ddots & \ddots
\end{pmatrix}
\] \tc{whose entries depend on $\kappa$}.
The decomposition \eqref{symUL} yields the relations
\begin{equation}\label{UL_entries1}
t_{j+1}(\kappa)=b_{j}-\kappa-t_j(\kappa)u_j^2(\kappa), \quad
t_{j+1}(\kappa)u_{j+1}(\kappa) = a_{j}, \quad j=0, 1, 2,\dots,
\end{equation}
which resemble \eqref{LU_entries1} but are fundamentally different. Namely, the first relation $t_{1}(\kappa)=b_{0}-\kappa-t_0(\kappa)u_0^2(\kappa)$ has a free parameter in it. To be definite, let $t_0=1$ and $u_0^2=1/s_0^*$.
Then, invoking \eqref{def:Rn} we get
\begin{equation}\label{UL_entries}
t_j(\kappa)=-a_{j-1}\frac{\hat{R}_{j}(\kappa)}{\hat{R}_{j-1}(\kappa)},\quad 
u_j(\kappa)=-
\frac{\hat{R}_{j-1}(\kappa)}{\hat{R}_{j}(\kappa)}.
\end{equation}
As before, we can rewrite \eqref{symUL} in the following manner
\begin{equation}\label{symUL_G}
J-\kappa I = U(\kappa) U^{\top}(\kappa)
\end{equation}
with
\begin{equation}\label{DefG}
U=\cU\cD^{1/2}=\begin{pmatrix}
u_0\sqrt{t_0} & \sqrt{t_1} & 0 & \cdots   \\
     0& u_1\sqrt{t_1} & \sqrt{t_2} &   \\
     0 & 0 & u_2\sqrt{t_2} & \ddots  \\
     \vdots & & \ddots & \ddots
\end{pmatrix},
\end{equation}
where one can take either of the two values of the square root of the complex number ${t_j}$. \tc{Now we will proceed similarly to the case of Christoffel transformation.} 
\begin{proposition}\label{symGeronimusTr}
Let $J$ be a Jacobi matrix corresponding to a positive-definite \tc{linear} functional $\mathcal{L}$. Also, let $\kappa\in\dC_{\pm}$ and let $s_0^*=\mathcal{L}^{-*}(1)\in\overline{\dC}_{\mp}\setminus\{0\}$. Then the factorization
\begin{equation}\label{GerFac}
J-\kappa I=UU^{\top}
\end{equation}
exists and the tridiagonal matrix 
\[
{J}_{G}(\kappa)=U^{\top}U+\kappa I
\]
corresponds to the quasi-definite functional $\mathcal{L}^{-*}$ 
\end{proposition}
\begin{proof}
Using \eqref{MatrixFormSec5} and \eqref{GerFac} one can see that 
\[
\begin{split}
(J-\kappa I)p(t)=(t-\kappa)p(t)\Rightarrow UU^{\top}p(t)=(t-\kappa)p(t)\Rightarrow\\
\Rightarrow (U^{\top}U)U^{\top}p(t)=(t-\kappa)U^{\top}p(t)\Rightarrow 
{J}_G(\kappa)U^{\top}p(t)=tU^{\top}p(t).
\end{split}
\]
Note that the entries of $U^{\top}p(t)$ are proportional to $P_n^{-*}(\kappa,z)$ and so they are orthogonal with respect to $\mathcal{L}^{-*}$. Thus,  ${J}_{G}(\kappa)$ corresponds to $\mathcal{L}^{-*}$.
\end{proof}
As before, the matrix $U$ also defines a closed linear operator on $\ell^2$ and it turns out that this operator, which will denote by $U$ as well, is bounded.
\begin{proposition}
Let $J$ be a Jacobi matrix corresponding to a positive-definite \tc{linear} functional $\mathcal{L}$. Also, let $\kappa\in\dC_{\pm}$ and let $s_0^*=\mathcal{L}^{-*}(1)\in\overline{\dC}_{\mp}\setminus\{0\}$. If $J$ is a bounded operator then $U$ is a bounded operator. 
\end{proposition}
\begin{proof}
The result follows from the reasoning given in Proposition \ref{Lbound} and the fact that the ratios of two consecutive polynomials $R_j$'s are bounded (see the proof of Proposition \ref{UboundP}).
\end{proof}
The functionals $\mathcal{L}$ and $\mathcal{L}^{-*}$ are related and so are the $m$-functions of $J$ and $J_G(\kappa)$.
\begin{proposition}
Let $J$ be a Jacobi matrix corresponding to a positive-definite \tc{linear} functional $\mathcal{L}$. Also, let $\kappa\in\dC_{\pm}$ and let $s_0^*=\mathcal{L}^{-*}(1)\in\overline{\dC}_{\mp}\setminus\{0\}$. If $J$ is bounded then 
\begin{equation}\label{formula:JandJg}
m(J_G(\kappa);z)=\frac{1}{s_0^*}\frac{m(J;z)}{z-\kappa}-\frac{1}{z-\kappa}
\end{equation}
for all complex numbers $z$, whose absolute value $|z|$ is sufficiently large.
\end{proposition}
\begin{proof}
This statement is another form of Proposition \ref{prop:JandJc}. Namely, substituting $J\to J_G$ and $J_C\to J$ into \eqref{formula:JandJc} we get \eqref{formula:JandJg}
by taking into account that
\[
b_0(J_G)-\kappa=\dfrac{1}{s_0^*}.
\]
\end{proof}
Now we are in the position to prove the result about the spectrum of $J_G(\kappa)$ that corresponds to the Geronimus transformation at $\kappa$.
\begin{theorem}
Let $J$ be a Jacobi matrix corresponding to a positive-definite \tc{linear} functional $\mathcal{L}$. Also, let $\kappa\in\dC_{\pm}$ and let $s_0^*=\mathcal{L}^{-*}(1)\in\overline{\dC}_{\mp}\setminus(\{0\}\cup\{m(J;\kappa)\})$.
If $J$ is a bounded operator then 
\[
\sigma({J}_G(\kappa))=\sigma(J)\cup\{\kappa\}.
\]
\end{theorem}
\begin{proof}
Analogously to the proof of Theorem \ref{Isospectral}, we get
\[
\sigma({J}_G(\kappa))\setminus\{\kappa\}=\sigma(J).
\]
Notice if we had $\kappa\in\rho({J}_G(\kappa))$ then $m(J_G(\kappa);z)$ given by \eqref{formula:JandJg} would be holomorphic at $\kappa$ but that could only happen when $s_0^*=m(J;\kappa)$, which is excluded by the assumptions.
\end{proof}
In conclusion note that it was not essential for us to start with a Jacobi matrix that corresponded to a positive-definite \tc{linear} functional. Most of the results can be adapted to just the case when the corresponding transformation exists and thus it gives rise to a number of iterations starting with a real Jacobi matrix It is even possible to proceed when the decomposition \eqref{DarFac} (or \eqref{GerFac}) does not exist, see \cite{DD11}). In principle, one can generalize some results that hold for real Jacobi matrices to the case of complex ones. For instance, one can mimic the operator proof of \cite[Theorem 2.1]{Simon04} for the complex case, which, in a way, was done in \cite{B01}. However, one has to replace the spectrum of an operator with the numerical range of the operator, which is essentially larger than the spectrum. At the same time, the results of this section along with the results derived in Sections 3 and 4 show that for $J_C$, $J_G$ and their iterations we can still get the results such as ratio asymptotic just outside the spectrum. Furthermore, for such complex Jacobi matrices one can construct reasonable functional calculus (e.g. see \cite[Theorem 7]{Barnes98}). 

\section{$R_I$- and $R_{II}$-recurrence relations}

In this section we will show how Darboux transformations can lead to $R_I$- and $R_{II}$-recurrence relations, which were introduced in \cite{IsmailMasson95} and were shown to be related to bi-orthogonal rational functions. \tc{Also, the results of this section provides with a different approach to the findings from \cite{DZh09} as well as extend some those.}

To begin with, note that the polynomial of degree $n+1$
\begin{equation}\label{f:quasi-orthogonal1}
 T_{n+1}(z)= P_{n+1}(z)+\tilde{A}_{n+1}P_n(z)   
\end{equation}
is orthogonal to the monomials $1$, $z$, \dots, $z^{n-1}$ for any choice of $\tilde{A}_{n+1}\in\dC$. Recall that such a polynomial is called quasi-orthogonal of order $1$. The following statement gives a relation among quasi-orthogonal polynomials, orthogonal polynomials, and kernel polynomials.
\begin{proposition}\label{thrm:R1_relation} Let $\{P_n(z)\}_{n=0}^\infty$ be a monic OPS with respect to a quasi-definite linear functional $\mathcal{L}$ and let $T_{n+1}(z)$ be a quasi-orthogonal polynomial of order $1$ that has the form \eqref{f:quasi-orthogonal1}. Assume that the entire sequence $\{P^*_n(\kappa,z)\}_{n=0}^\infty$ of the corresponding kernel polynomials exists for some $\kappa\in\dC_\pm$. Then there
exist unique sequences of constants $\alpha_n$ and $\beta_n$ such that
\begin{equation}\label{R1-relation}
    T_{n+1}(z)-(z-\alpha_n)P_n(z)+\beta_n(z-\kappa)P^*_{n-1}(\kappa,z)=0.
\end{equation}
\end{proposition}
\begin{proof}
Using the definitions of kernel and quasi-orthogonal polynomials yields
\begin{multline}
    T_{n+1}(z)-(z-\alpha_n)P_n(z)+\beta_n(z-\kappa)P^*_{n-1}(\kappa,z)=\\
    =P_{n+1}(z)\tc{-}\left( z-\tilde{A}_{n+1}-\alpha_n-\beta_n\right)P_n(z)-\beta_n \frac{P_n(\kappa)}{P_{n-1}(\kappa)}P_{n-1}(z),
\end{multline}
where $P_{n-1}(\kappa)\ne 0$ since the kernel polynomials exist for any nonnegative integer $n$. Next, one can rewrite \eqref{def:rec2} as follows  
\[
P_{n+1}(z)-(z-c_{n+1})P_n(z)+\lambda_{n+1}P_{n-1}(z)=0.
\]
Hence, putting
\[
\beta_n=-\frac{\lambda_{n+1}P_{n-1}(\kappa)}{P_n(\kappa)} , \quad \alpha_n=c_{n+1}-\tilde{A}_{n+1}+\lambda_{n+1}\frac{P_{n-1}(\kappa)}{P_n(\kappa)}
\]
one arrives at the desired relation \eqref{R1-relation}.
\end{proof} 
In particular, we can apply Proposition \ref{thrm:R1_relation} to the case where the polynomials $T_{n+1}(z)$ correspond to the Geronimus transformation at some point.
\begin{corollary}\label{col:R_I}
 Let $\{P_n(z)\}_{n=0}^\infty$ be a monic OPS with respect to a \tc{positive-definite} linear functional $\mathcal{L}$ and let $\{P^*_n(\kappa_1,z)\}_{n=0}^\infty$ be the corresponding kernel polynomials for some $\kappa_1\in\dC_\pm$. Assume that $\kappa_2\in\dC_\pm$ is chosen so that the entire sequence $\{P^{-*}_n(\kappa_2, z)\}_{n=0}^\infty$ corresponding to the Geronimus transformation at $\kappa_2$ exists. Then we have that
 \begin{equation}\label{f:R1}
     P^{-*}_{n+1}(\kappa_2,z)-(z-\alpha_n)P_n(z)+\beta_n(z-\kappa_1)P^*_{n-1}(\kappa_1,z)=0,
 \end{equation}
 where 
 \[
 \beta_n=-\frac{\lambda_{n+1}P_{n-1}(\kappa_1)}{P_n(\kappa_1)} , \quad \alpha_n=c_{n+1}+\frac{R_{n+1}(\kappa_2)}{R_n(\kappa_2)}+\lambda_{n+1}\frac{P_{n-1}(\kappa_1)}{P_n(\kappa_1)}.
 \]
\end{corollary}
\begin{proof}
From Theorem \ref{thm:RegGer} we know that 
\[
P_n^{-*}(\kappa_2,z)= P_n(z)+A_nP_{n-1}(z),
\] 
where 
\[
A_n=-\frac{s_0^*P_{n+1}(\kappa_2)+Q_{n+1}(\kappa_2)}{s_0^*P_{n}(\kappa_2)+Q_{n}(\kappa_2)} = -\frac{R_{n+1}(\kappa_2)}{R_n(\kappa_2)}.
\]
Therefore, setting $\tilde{A}_n=A_n$ in Proposition \ref{thrm:R1_relation} gives \eqref{f:R1}.
\end{proof} 
The explicit formulas for the coefficients of \eqref{f:R1} also yield the following.
\begin{corollary}
If in addition to the assumptions of Corollary \ref{col:R_I} we assume that $\{P_n(z)\}_{n=0}^\infty$ is in the Nevai class $\mathcal{N}(a,c)$, then the sequences $\alpha_n$ and $\beta_n$ are convergent and
\[
\lim_{n\to\infty}\alpha_n=-\dfrac{a}{f(\kappa_1)}, \quad \lim_{n\to\infty}\beta_n=c+f(\kappa_2)+\dfrac{a}{f(\kappa_1)},
\]
where $f$ is explicitly given in \tc{\eqref{def:RatioAsympF}}.
\end{corollary}
At this point we can easily show the connection to $R_I$-recurrence relations. Let $\mathcal{L}$ be a positive-definite \tc{linear} functional generated by a probability measure $d\mu$, whose support is contained in the finite interval $[a,b]$, that is,
\[
\mathcal{L}(p(t))=\int_a^bp(t)\,d\mu(t)
\]
for any polynomial $p(t)$. Next, assume that we have an infinite sequence $\kappa_1$, $\kappa_2$, \dots \, of distinct numbers in $\dC\setminus\dR$ such that each functional in the sequence 
\[
\mathcal{L}, \mathcal{L}^{-*}, \mathcal{L}^{-**}, \mathcal{L}^{-3*}=\mathcal{L}^{-***}=(\mathcal{L}^{-**})^{-*}, \dots
\]
is quasi-definite. In other words, the polynomial
\[
P_n^{-m*}(\kappa_1,\dots,\kappa_m,z)
\]
that corresponds to the $m$-th iteration of Geronimus transformation is correctly defined for any nonnegative integers $n$ and $m$. \tc{One of such choices could be $\kappa_1=k_1$, $\kappa_2=\overline{k_1}$, $\kappa_3=k_2$, $\kappa_4=\overline{k_2}$, \dots, in which case Theorem \ref{IteratedGeronimus} guarantees the existence of $P_n^{-m*}(\kappa_1,\dots,\kappa_m,z)$ for any nonnegative integers $n$ and $m$.} Thus, we have a table of polynomials. By looking at the diagonal of this table
\[
\Pi_n(z)=P_n^{-n*}(\kappa_1,\dots,\kappa_n,z)
\]
one can notice that its elements satisfy the relation
\[
\Pi_{n+1}(z)-(z-\hat{\alpha}_n)\Pi_n(z)+\hat{\beta}_n(z-\kappa_{n-1})\Pi_{n-1}(z)=0
\]
with some sequence $\hat{\alpha}_n$ and $\hat{\beta}_n$. The latter relation is exactly an $R_I$-recurrence relation (see \cite{IsmailMasson95} for the definition). Moreover, every diagonal in the table of polynomials $P_n^{-m*}$ satisfies a similar relation.

To get to the next level, observe that 
\begin{equation}\label{f:quasi-orthogonal2}
 S_{n+1}(z)= P_{n+1}(z)+\tilde{C}_{n}P_n(z)+\tilde{D}_nP_{n-1}(z)  
\end{equation}
is orthogonal to the monomials $1$, $z$, \dots, $z^{n-2}$ for any choice of $\tilde{D}_{n}, \tilde{C}_n \in \mathbb{C}$. Such a polynomial is called quasi-orthogonal of order $2$ and in this case we also have a relation that involves a quasi-orthogonal polynomial of order $2$ and an iterated kernel polynomial.
\begin{theorem}\label{thrm:R_2}
Let $\{P_n(z)\}_{n=0}^\infty$ be a monic OPS with respect to a \tc{positive-definite} linear functional $\mathcal{L}$ and let $\kappa_1\in \mathbb{C}_{\pm}$ and $\kappa_2 \in \mathbb{C}_{\mp}$. Also, let  $\{P^{**}_n(\kappa_1,\kappa_2,z)\}_{n=0}^\infty$ be the corresponding iterated kernel polynomials and let $S_{n+1}(z)$ be a a quasi-orthogonal polynomial of order $2$ whose degree is $n+1$. Then there
exist unique sequences of constants $\gamma_n$, $\upsilon_n$ and $\rho_n$ such that
\begin{equation}\label{R2-relation}
    S_{n+1}(z)-(\rho_nz-\gamma_n)P_n(z)+\upsilon_n(z-\kappa_1)(z-\kappa_2)P^{**}_{n-1}(\kappa_1, \kappa_2,z)=0.
\end{equation}
\end{theorem}
\begin{proof} Using the definition of the entries, the left-hand side of \eqref{R2-relation} takes the form  
\begin{multline}\label{eq:SnPn}
    S_{n+1}(z)-(\rho_nz-\gamma_n)P_n(z)+\upsilon_n(z-\kappa_1)(z-\kappa_2)P^{**}_{n-1}(\kappa_1, \kappa_2,z) = \\
    P_{n+1}(z)+\tilde{C}_nP_n(z)+\tilde{D}_nP_{n-1}(z)-\rho_nzP_n(z)+\gamma_nP_n(z)
    +\upsilon_n P_{n+1}(z)\\-\upsilon_n\frac{P_{n+1}(\kappa_1)}{P_n(\kappa_1)}P_n(z)-\upsilon_n \frac{P^*_n(\kappa_1, \kappa_2)}{P^*_{n-1}(\kappa_1, \kappa_2)}P_n(z)+\upsilon_n\frac{P^*_n(\kappa_1, \kappa_2)}{P^*_{n-1}(\kappa_1, \kappa_2)}\frac{P_n(\kappa_1)}{P_{n-1}(\kappa_1)}P_{n-1}(z).
\end{multline}
Since $zP_n(z)=P_{n+1}(z)+c_{n+1}P_n(z)+\lambda_{n+1}P_{n-1}(z)$, substituting we have the right-hand side of equation \eqref{eq:SnPn} equals
\begin{multline*}
    (1-\rho_n+\upsilon_n)P_{n+1}(z) + \left( \tilde{C}_n-\rho_nc_{n+1}+\gamma_n-\upsilon_n\left(\frac{P_{n+1}(\kappa_1)}{P_n(\kappa_1)}+\frac{P^*_n(\kappa_1, \kappa_2)}{P^*_{n-1}(\kappa_1, \kappa_2)} \right)\right)P_n(z)\\
    +\left( \tilde{D}_n-\rho_n\lambda_{n+1}+\upsilon_n\frac{P^*_n(\kappa_1, \kappa_2)}{P^*_{n-1}(\kappa_1, \kappa_2)}\frac{P_n(\kappa_1)}{P_{n-1}(\kappa_1)}\right)P_{n-1}(z).
\end{multline*}
Setting this equal to zero and using the linear independence of $P_{n+1}(z), P_n(z)$ and $P_{n-1}(z)$, we see that the sequences
\[\upsilon_n= \frac{\lambda_{n+1}-\tilde{D}_n}{\frac{P^*_n(\kappa_1, \kappa_2)}{P^*_{n-1}(\kappa_1, \kappa_2)}\frac{P_n(\kappa_1)}{P_{n-1}(\kappa_1)}-\lambda_{n+1}}, \quad \rho_n=1+\upsilon_n\]  and \\
\[ \gamma_n=\rho_nc_{n+1}+\upsilon_n\left(\frac{P_{n+1}(\kappa_1)}{P_n(\kappa_1)}+\frac{P^*_{n}(\kappa_1, \kappa_2)}{P^*_{n-1}(\kappa_1, \kappa_2)} \right)-\tilde{C}_n\] satisfy equation \eqref{R2-relation}. Let us stress here that $\upsilon_n$ is well defined. Indeed, we have that 
\[
\frac{P^*_n(\kappa_1,\kappa_2)}{P^*_{n-1}(\kappa_1, \kappa_2)}\frac{P_n(\kappa_1)}{P_{n-1}(\kappa_1)}=\frac{P_{n+1}(\kappa_2)P_n(\kappa_1)-P_{n+1}(\kappa_1)P_n(\kappa_2)}{P_n(\kappa_2)P_{n-1}(\kappa_1)-P_n(\kappa_1)P_{n-1}(\kappa_2)},
\]
where $P_n(\kappa_2)P_{n-1}(\kappa_1)-P_n(\kappa_1)P_{n-1}(\kappa_2)\ne 0$ because $P_n^{**}(\kappa_1,\kappa_2,z)$ exists. Then one can see that
\begin{align*}
    \frac{P_{n+1}(\kappa_2)P_n(\kappa_1)-P_{n+1}(\kappa_1)P_n(\kappa_2)}{P_n(\kappa_2)P_{n-1}(\kappa_1)-P_n(\kappa_1)P_{n-1}(\kappa_2)}
    &=\lambda_{n+1}\left( \frac{\sum_{j=0}^n \frac{P_j(\kappa_2)P_j(\kappa_1)}{\lambda_1\dots \lambda_{j+1}}}{\sum_{j=0}^{n-1}\frac{P_j(\kappa_2)P_j(\kappa_1)}{\lambda_1\dots \lambda_{j+1}}}\right)\\
    &= \lambda_{n+1}\left(1+ \frac{P_n(\kappa_2)P_n(\kappa_1)}{{\sum_{j=0}^{n-1}\frac{P_j(\kappa_2)P_j(\kappa_1)}{\lambda_1\dots \lambda_{j+1}}}}\right)\\
    &\neq \lambda_{n+1},
\end{align*} where the last line follows from the fact that $\kappa_1, \kappa_2 \in \mathbb{C}_{\pm}$ and $P_n(z)$ has only real zeros for all $n=1, 2, \dots$.
\end{proof}

Since $P^{-**}_n(\kappa_2,\overline{\kappa}_2,z)$ has the form \eqref{f:quasi-orthogonal2}, one can consider the following particular case of Theorem \ref{thrm:R_2} \tc{that gives a different derivation and another proof of the recurrence relations obtained in \cite{DZh09}.} 

\begin{corollary}\label{cor:R_2v}
 Let $\{P_n(z)\}_{n=0}^\infty$ be a monic OPS with respect to a \tc{positive-definite} linear functional $\mathcal{L}$ of the form \eqref{functional:IntRepr} and let $\kappa_1, \kappa_2\in\dC_+$. Also, let $\{P^{**}_n(\kappa_1,\overline{\kappa_1},z)\}_{n=0}^\infty$ be the corresponding iterated kernel polynomials and let $\{P^{-**}_n(\kappa_2,\overline{\kappa_2},z)\}_{n=0}^\infty$ be the polynomials corresponding to 
 \[
 s_0^*=\int_{a}^{b}\frac{d\mu(t)}{t-\kappa_2}, \quad
 s_0^{**}=\int_{a}^{b}\frac{d\mu(t)}{|t-\kappa_2|^2}.
 \]
 Then we have the following relation
 \begin{equation}\label{eq:forR2}
 P^{-**}_{n+1}(\kappa_2,\overline{\kappa_2},z)-((1+\upsilon_n)z-\gamma_n)P_n(z)+\upsilon_n(z-\kappa_1)(z-\overline{\kappa_1})P^{**}_{n-1}(\kappa_1, \overline{\kappa_1},z)=0.  
\end{equation}
\end{corollary}

According to formula \eqref{eq:bilinear}, the choice of the coefficients $s_0^*$ and $s_0^{**}$ guarantees the existence of $\{P^{-**}_n(\kappa_2,\overline{\kappa_2},z)\}_{n=0}^\infty$ and, in fact, the sequence is an OPS with respect to the measure
\[
\dfrac{d\mu(t)}{|t-\kappa_2|^2}.
\]
Using Theorem \ref{IteratedGeronimus} and the formulas derived in the proof of Theorem \ref{thrm:R_2}, one can also prove that the sequences $\upsilon_n$ and and $\gamma_n$ are convergent provided that $\{P_n(z)\}_{n=0}^\infty$ is in a Nevai class.

To finalize the relation to $R_{II}$-recurrence relations, consider an infinite sequence $\kappa_1$, $\kappa_2$, \dots of distinct numbers in $\dC_+$. Next, let $\mathcal{P}_n(z)$ denote the $n$-th orthogonal polynomial with respect to the measure
\[
\dfrac{d\mu(t)}{|t-\kappa_1|^2|t-\kappa_2|^2\dots |t-\kappa_n|^2},
\]
that is,
\begin{equation}\label{VarOrt}
 \int_a^b\mathcal{P}_n(t)t^m\dfrac{d\mu(t)}{|t-\kappa_1|^2|t-\kappa_2|^2\dots |t-\kappa_n|^2} \tc{ =0}, \quad m=0,1,\dots, n-1.   
\end{equation}
Therefore, the polynomials $\mathcal{P}_n(z)$'s are orthogonal with respect to the varying measure and, at the same time, each of them is derived by consecutively applying the Geronimus transformations described in Corollary \ref{cor:R_2v} at the corresponding points. \tc{For example, if we start with the Chebyshev polynomials of the first or second kind, the corresponding polynomial $\mathcal{P}_n$ coincides with the extremal polynomial introduced by Bernstein (see \cite[pp. 249-254]{A65}).} Applying \eqref{eq:forR2} to this \tc{particular choice of transformations we get}
\begin{equation*}
 \mathcal{P}_{n+1}(z)-((1+\hat{\upsilon}_n)z-\hat{\gamma}_n)\mathcal{P}_n(z)+\hat{\upsilon}_n(z-\kappa_{n-1})(z-\overline{\kappa}_{n-1})
 \mathcal{P}_{n-1}(z)=0.  
\end{equation*}
The latter relation is an $R_{II}$-recurrence relation and it was shown to be related to multipoint Pad\'e approximants (see \cite{DZh09}, \tc{where it is also shown how exactly this relation corresponds to a pair of Jacobi matrices}). \tc{Finally, to demonstrate the relation to orthogonal rational functions note that \eqref{VarOrt} can be rewritten as
\[
\int_a^b
\dfrac{\mathcal{P}_n(t)}{(t-\kappa_1)(t-\kappa_2)\dots (t-\kappa_n)}
\dfrac{t^m}{(t-\overline{\kappa}_1)(t-\overline{\kappa}_2)\dots (t-\overline{\kappa}_n)} d\mu(t)=0,
\]
which implies that the rational function
\[
R_n(t)=\dfrac{\mathcal{P}_n(t)}{(t-\kappa_1)(t-\kappa_2)\dots (t-\kappa_n)}
\]
is orthogonal to $1/(t-\kappa_j)$, $j=1$, \dots, $n$ with respect to $d\mu$.}

\vspace{6mm}

\noindent {\bf Acknowledgments.} This research was supported by the NSF DMS grant 2008844 and by the University of Connecticut Research Excellence Program.

\end{document}